\documentclass[11pt,reqno]{article}
\usepackage{amsmath}
\usepackage{mathrsfs}
\usepackage{amsfonts}
\usepackage{amssymb}

\usepackage{amssymb}
\usepackage{amsthm}
\usepackage{graphicx}              
\usepackage{amsmath}               
\usepackage{amsfonts}              
\usepackage{amsthm}                
\usepackage{setspace}
\usepackage{epstopdf}
\usepackage{epsfig}

\newtheorem{theorem}{Theorem}[section]
\newtheorem{proposition}{Proposition}[section]
\newtheorem{lemma}{Lemma}[section]
\newtheorem{corollary}{Corollary}[section]
\newtheorem{remark}{Remark}[section]
\newtheorem{definition}{Definition}[section]

\renewcommand{\thefootnote}{\fnsymbol{footnote}}
\makeatletter

\newcommand{\Rmnum}[1]{\expandafter\@slowromancap\romannumeral #1@}
\makeatother

\allowdisplaybreaks
\numberwithin{equation}{section}

\title{
Asymptotic behaviors of Landau-Lifshitz flows from $\Bbb R^2$ to K\"ahler manifolds}

\author{ Ze Li \qquad Lifeng Zhao}

\date{}

\begin{document}

\maketitle

\renewcommand{\thefootnote}{\fnsymbol{footnote}}
\footnotetext{\hspace*{-5mm}
\begin{tabular}{@{}r@{}p{155mm}@{}}
&Wu Wen-Tsun Key Laboratory of Mathematics, Chinese Academy of Sciences
 and Department of Mathematics, University of Science and Technology of China, Hefei 230026, \ Anhui, \ China.
Ze Li, E-mail: u19900126@163.com. \mbox{  }Lifeng Zhao, E-mail: zhaolf@ustc.edu.cn.
\end{tabular}}

\noindent{\bf Abstract}
In this paper, we study the asymptotic behaviors of finite energy solutions to the Landau-Lifshitz flows from $\Bbb R^2$ into K\"ahler manifolds. First, we prove that the solution with initial data below the critical energy converges to a constant map in the energy space as $t\to \infty$ for the compact Riemannian surface targets. In particular, when the target is a two dimensional sphere, we prove that the solution to the Landau-Lifshitz-Gilbert equation with initial data having an energy below $4\pi$ converges to some constant map in the energy space. Second, for general compact K\"ahler manifolds and initial data of an arbitrary finite energy, we obtain a bubbling theorem analogous to the Struwe's results on the heat flows.

\noindent{\bf MR Subject Classification:} XXXx.

\bigskip

\section{Introduction}
In this article, we consider the two dimensional Landau-Lifshitz (LL) equation:
\begin{equation}\label{1}
\begin{cases}
u_t =\sum\limits_{i= 1}^2 \alpha\nabla_{x_i} \partial_{x_i}u-\beta J(\nabla_{x_i} \partial_{x_i}u)\\
u(0) = u_0,
\end{cases}
\end{equation}
where $u(x,t):\Bbb R^2\times [0,\infty)\to \mathcal{N}$, $\big(\mathcal{N},J,h\big)$ is a K\"ahler manifold, $\nabla_x$ is the induced connection by $u$, $\partial_{x_i}u=u_*(\frac{\partial}{\partial_{x_i}})$, $\partial_tu=u_*(\frac{\partial}{\partial_t})$. $\alpha\ge0$ is called the Gilbert constant. When $\alpha=0$, (\ref{1}) is called the Schr\"odinger flow. When $\beta=0$, $\alpha>0$, it reduces to the heat flows of harmonic maps. The energy of $u$ is given by
$$
E(u)=\frac{1}{2}\int_{\Bbb R^2}|\nabla u|^2_{u^*h}dx.
$$
If the starting manifold of $u$ is a general Riemannian manifold $\mathcal{M}$, the term $\Sigma^2_{i=1}\nabla_{x_i} \partial_{x_i}u$ in (\ref{1}) should be replaced by $\tau(u)\triangleq tr_g(\nabla du)$, where $g$ is the metric in $\mathcal{M}$.
If $\mathcal{N}$ is the two dimensional sphere $\Bbb S^2$, (\ref{1}) can be written as the following system
 \begin{equation}\label{191}
\begin{cases}
u_t =\alpha\Delta u+\alpha|\nabla u|^2u-\beta u\times\Delta u,\\
u(0) = u_0,
\end{cases}
\end{equation}
where $u:[0,T]\times \Bbb R^2\to \Bbb R^3$ satisfies $|u|=1$. (\ref{191}) is called the Landau-Lifshitz-Gilbert equation when $\alpha>0$ and the Heisenberg
equation when $\alpha=0$.

For the different choices of the target $\mathcal{N}$, (\ref{1}) is related to various theories in mechanics and physics. For targets of Riemannian surfaces such as $\Bbb S^2$ or $\Bbb H^2$, (\ref{1}) is related to the gauge theories. For the target of two-dimensional sphere, (\ref{1}) describes the evolution of static as well as dynamic properties of magnetization (\cite{LL}). Moreover, the LL flow with a sphere target arises in the classical continuous isotropic Heisenberg spin model, or the long wave length limit of the isotropic Heisenberg ferromagnet.

There exist plenty of results on the well-posedness and dynamic behaviors of the Landau-Lifshitz equations. We recall the following non-exhaustive list of works. First we consider the well-posedness theories. The first mathematical work on the Heisenberg equation is done by Sulem, Sulem and Bardos \cite{SSB} who proved the local well-posedness on $\Bbb R^d$ ($d\ge 1$). Zhou, Guo, Tan \cite{ZGT} studied the global well-posedness problem by the viscosity method. Ding and Wang \cite{DW} and McGahagan \cite{Mc} proved the local existence and uniqueness of solutions from closed Riemannian manifolds or $\Bbb R^d$ into compact K\"ahler targets in some Sobolev spaces. Chang, Shatah and Uhlenbeck \cite{CSU} proved the global well-posedness of smooth solutions from $\Bbb R$ or $\Bbb R^2$ into compact Riemannian surfaces under additional small assumptions on the data. Rodnianski, Rubinstein and Staffilani \cite{RRS} obtained the global well-posedness of Schr\"odinger flows from $\Bbb R$ into K\"ahler manifolds and flows from $\Bbb S^1$ to Riemannian surfaces.
For maps from $\Bbb R^d$ ($d\ge4$) into $\Bbb S^2$ with initial data of small critical Sobolev norms, Bejenaru, Ionescu and Kenig \cite{BIK} proved the global well-posedness. Later the $d=2,3$ cases were proved in Bejenaru, Ionescu, Kenig and Tataru \cite{BIKT1}. For the dissipative ($\alpha>0$) and the $\Bbb S^2$ target case, there are a lot of works on the global existence of weak solutions and partial regularity theory for Landau-Lifshitz-Gilbert equations, for instance \cite{M,GH,CF,AS}.

The dynamic behavior of the LL flow is known in the equivariant case and the small data case.  Finite time blow up solutions near the harmonic maps were constructed by Chang, Ding, Ye \cite{CDY} for the 2D heat flows,  Merle, Raphael, Rodnianski \cite{MPR2} and Perelman \cite{P} for 1-equivariant Schr\"odinger maps from $\Bbb R^2$ to $\Bbb S^2$.  The asymptotic stability of harmonic maps under the LL flow in the equivariant case was proved by Gustafson, Kang, Tsai \cite{GKT1,GKT2} and Gustafson, Nakanishi, Tsai \cite{GNT}.  For equivariant initial data with energy below the ground state,  Bejenaru, Ionescu, Kenig and Tataru \cite{BIKT2, BIKT3} proved the well-posedness and the scattering in the gauge sense when the target is $\Bbb S^2$ or $\Bbb H^2$.

The dynamic behavior for general initial data has been studied for the heat flow to some extent. For the LL flow even for the Landau-Lifshitz-Gilbert equation, merely some partial results were obtained. One of the typical results on the dynamic behaviors of the heat flow is the bubble tree convergence which has been intensively studied for instance Jost \cite{J}, Parker \cite{P}, Qing \cite{Q}. The bubble tree convergence means the solution will evolve as a superposition of a harmonic map and some rescaled and translated bubbles along some time sequence as $t\to\infty$. The corresponding result for the Landau-Lifshitz-Gilbert equation was proved by Harpe \cite{H}. Notice that whether the bubbles and the harmonic map are the same for different time sequences is still largely open even in the heat flow case. Thus more efforts should be paid to understand the whole picture of the dynamic behaviors. In this paper, we consider initial data of energy below the critical energy. In our sequel papers, we will continue our works on the dynamic behaviors of (\ref{1}).

The global well-posedness in our case can be obtained by the Struwe's bubbling arguments on the heat flow, see Theorem \ref{1}. The new difficulty is the non-compactness of $\Bbb R^2$ and the second derivative term with the complex structure. The non-compactness will be overcome by an outer ball energy estimate. In order to avoid the obstacle to the energy arguments caused by the second derivative term with the complex structure, we fully use the skew-symmetry of the symplectic form to obtain some cancellation of the high derivative terms. We remark that Theorem \ref{1} below also yields a rough description of the dynamic behavior as $t\to\infty$ for initial data below the threshold. In fact, Theorem \ref{1} implies the LL flow converges locally to a constant map up to some subsequence, some scaling and some translation. The convergence for all time in the energy topology requires additional efforts. This is then solved by Theorem \ref{zaq}.

After proving the global well-posedness, in order to get the complete dynamic picture below the threshold, we apply the techniques developed in the semilinear and geometric dispersive PDEs, especially the method of induction on energy and geometric renormalizations, see for instance Bourgain \cite{B} and Chapter 6 of Tao \cite{T1}. The proof involves three essential ingredients. 
First, because of the dissipative nature of (\ref{1}), we can gain a prior $L^2_{t,x}$ space-time bound for the field $\tau(u)$. Meanwhile, the induction on energy argument gives an $L^4_{t,x}$ space-time bound for $\nabla u$. Thus we obtain the $L^2_{t,x}$ norm of $|\nabla^2 u|$. Second, rewriting (\ref{1}) in the Coulomb gauge yields a Ginzburg-Landau type system coupled with a Poisson system for the differential fields and the connection coefficients. The Poisson system gives a useful bound for the connection coefficients by the prior $L^2_{t,x}$ norm of $|\nabla^2 u|$. Finally, the decay of the energy follows by applying the Strichartz estimates to the Ginzburg-Landau equation for the differential fields.

The main results of this paper are the following two.
For general compact K\"ahler targets and general data, we obtain the almost regularity and bubbling theorem:
\begin{theorem}\label{aaa.1}
Let $(\mathcal{N},h,J)$ be a compact K\"ahler manifold, $\alpha>0$, $\beta\in \Bbb R$. For any data $u_0\in W^{1,2}(\Bbb R^2;\mathcal{N})$, there exists a weak solution in $L^{\infty}([0,\infty);W^{1,2}(\Bbb R^2;\mathcal{N}))$ to (\ref{1}), which is regular on $\Bbb R^2\times (0,\infty)$ with the exception of finitely many points $(x_l,T_l)$, $1\le l\le L$, characterized by
$$
\mathop {\lim \sup }\limits_{t \to T_l} \int_{B(x_l,R)}|\nabla u(t,y)|^2dy> {\varepsilon _1}, \mbox{  }\mbox{  }{\rm{for}}\mbox{  }{\rm{all}}\mbox{  }R\in(0,1],
$$
where $\varepsilon_1$ is some positive constant depending only on $\mathcal{N}$.
Furthermore, for any fixed pair $(x_l,T_l)$ there exist sequences $t_m\to T_l$, $x_m\to x_l$, $R_m\to 0$ and a harmonic map $u_{\infty}\in C^{\infty}(\Bbb R^2;\mathcal{N})$ such that
$$
u(t_m, R_mx+x_m) \to {u_\infty } \mbox{  }\mbox{  }\mbox{  }{\rm{locally}}\mbox{  }{\rm{in}}\mbox{  }W^{2,2}(\Bbb R^2;\mathcal{N}).
$$
\end{theorem}

To state Theorem 1.2, we define the critical energy as follows
\begin{align}\label{zas33}
E_*=\inf\{E:E=E(Q(x)),\mbox{  }{\rm{where}}\mbox{  }Q(x):\Bbb R^2\to \mathcal{N}\mbox{ }{\rm{is}}\mbox{ }{\rm{a}}\mbox{  } {\rm{harmonic}}\mbox{  }{\rm{map}}\mbox{  } {\rm{with}} \mbox{  }E(Q)>0\}.
\end{align}
We make the convention that $E_*=\infty$ if there is no non-trivial harmonic map from $\Bbb R^2$ to $\mathcal{N}$ with finite energy.
For compact Riemannian surfaces, we have
\begin{theorem}\label{zaq}
Let $(\mathcal{N},h,J)$ be a compact Riemannian surface, $\alpha>0$, $\beta\in \Bbb R$. The LL flow with $u_0\in W^{1,2}(\Bbb R;\mathcal{N})$ satisfying $E(u_0)<E_*$ admits a global solution $u\in L^{\infty}([0,\infty);W^{1,2}(\Bbb R; \mathcal{N}))$. Moreover, $u(t,x)$ converges to a constant map as $t\to\infty$ in the energy space, namely
$$
\mathop {\lim }\limits_{t \to \infty } E(u(t)) = 0.
$$
\end{theorem}

\noindent{\bf Remark 1.1} It is known that (Schoen and Yau \cite{SY}) $E_*=\infty$  if the sectional curvature of $\mathcal{N}$ is non-positive. Therefore, Theorem \ref{zaq} shows all the solutions of (\ref{1}) with finite energy decay to zero if $\mathcal{N}$ has a non-positive sectional curvature. Typical examples for compact Riemannian surfaces with non-positive curvature are Bolza surface, Klein quartic, Bring's surface and Macbeath surface. For general compact targets, $E_*$ is always strictly positive and we have an explicit low bound for $E_*$ by using the upper bound of the Riemannian curvature of $\mathcal{N}$ (see Lemma \ref{cri}). And it is known that $E_*=4\pi$ when $\mathcal{N}$ is a two-dimensional sphere. Considering that the $\Bbb S^2$ target is of special physical importance, we state the corresponding result of the $\Bbb S^2$ target as a corollary below.

\begin{corollary}\label{123.1}
Let $\alpha>0$, $\beta\in \Bbb R$. The Landau-Lifshitz-Gilbert equation (\ref{191}) with $u_0\in W^{1,2}(\Bbb R;\mathcal{N})$ satisfying $E(u_0)<4\pi$ admits a global solution and
$$
\mathop {\lim }\limits_{t \to \infty } E(u(t)) = 0.
$$
\end{corollary}

In what follows we give a brief overview of the paper. In Section 2, we prove Theorem \ref{zaq} under the assumption that Theorem \ref{1} has been proved. In Section 3, we prove Theorem \ref{aaa.1}.

\noindent{\bf{Function Spaces and Notations}}
The covariant derivative in $T\mathcal{N}$ is denoted by $\overline{\nabla}$, the covariant derivative induced by $u$ in $u^*(T\mathcal{N})$ is denoted by $\nabla$, the usual derivative for scalar functions is denoted by $D$. The Riemannian curvature tension of $\mathcal{N}$ is denoted by $\mathbf{R}$.
We use both the extrinsic and intrinsic Sobolev norms for maps from $\Bbb R^2$ to $\mathcal{N}$. In general, the two norms are not equivalent.
The extrinsic Sobolev spaces are defined as follows.  Let $\mathcal{N}$ be a closed submanifold of the Euclidean space $\Bbb R^m$. For a map $u$ from $\Bbb R^2$ to $\mathcal{N}$, we use the extrinsic expression $u=(u^1,...,u^m)$, where $u^i$ is defined as a function from $\Bbb R^2$ to $\Bbb R$, and $(u^1,...,u^m)\in \mathcal{N}, a.e.$. We say $u\in W^{k,p}(\Bbb R^2;\mathcal{N})$ if there is a point $Q\in \mathcal{N}$ such that $\|u^i-Q^i\|_{W^{k,p}(\Bbb R^2;\Bbb R)}<\infty$, for all $i\in\{1,...,m\}$, and $u(x)\in\mathcal{N}$ for a.e. $x\in\Bbb R^2$. The norm of $W^{k,p}$ is defined by
$$
\|u\|_{W^{k,p}}=\sum^m_{i=1}\|u^i-Q^i\|_{W^{k,p}(\Bbb R^2;\Bbb R)}.
$$
We also introduce the intrinsic semi-norm for maps belonging to $W^{k,p}(\Bbb R^2;\mathcal{N})$:
\begin{align*}
{\left\| u \right\|_{{\mathcal{W}^{k,p}}({\Bbb R^2};{\mathcal{N}})}} = \sum\limits_{\left\{ {{j_1},...,{j_k}} \right\} \subset \left\{ {1,2} \right\}} {{{\left\| {{{\left| {{\nabla _{{x_{{j_1}}}}}...{\nabla _{{x_{{j_{k - 1}}}}}}{\partial _{{x_k}}}u} \right|}_{{u^ * }h}}} \right\|}_{{L^p}}}}.
\end{align*}
For convenience, we denote
$\left| {{\nabla ^2}u} \right| = \big(\sum\limits_{\left\{ {{j_1},{j_2}} \right\} \subset \left\{ {1,2} \right\}} {{{\left| {{\nabla _{{x_{{j_1}}}}}{\partial _{{x_{{j_2}}}}}u} \right|}^2_{{u^ * }h}}}\big)^{1/2}.$
We will usually use Kato's inequality, which says in the distribution sense,
$$|\partial_x|\nabla u||\le |\nabla^2 u|.$$

\section{The proof of Theorem \ref{zaq}}
Theorem \ref{aaa.1} will be proved in Section 3. In this section, we prove Theorem \ref{zaq} by postulating Theorem \ref{aaa.1}. We emphasize that the proof of Theorem \ref{aaa.1} is independent of the results in this section. For convenience, we first summarize the well-posedness theory obtained in Section 3.
We recall the following notations:\\
$(1)$ (Local energy)
$$E(u(t);B_R(x))=\frac{1}{2}\int_{B_R(x)}|\nabla u(t,y)|^2dy;$$
$(2)$ (weak solution class)
\begin{align*}
 &Y([0,T] \times {\mathbb{R}^2}) \\
 &\triangleq\left\{ {u:[0,T] \times {\mathbb{R}^2} \to {\mathbb{R}^m},u(t,x) \in \mathcal{N},a.e.\left| \begin{array}{l}
 u \in C([0,T];{L^2}({\mathbb{R}^2})),\nabla u \in {L^\infty }([0,T];{L^2}({\mathbb{R}^2})) \\
 {\nabla ^2}u \in {L^2}([0,T] \times {\mathbb{R}^2}),{\partial _t}u \in {L^2}([0,T] \times {\mathbb{R}^2}) \\
 \end{array} \right.} \right\}.
\end{align*}

\begin{proposition}\label{energy}
Define the solution class $\mathcal{H}(I\times\Bbb R^2)$ as the set of all weak solutions to (\ref{1}) which satisfy for all $R>0$, $(s_1,s_2)\subset I$,
\begin{align}
&(i)\mbox{  }u\in Y(I\times\Bbb R^2) ;\nonumber\\
&(ii)\mbox{ }\alpha \int_{{s_1}}^{{s_2}} {\left\| {{\partial _s}u} \right\|_{L_x^2}^2}ds + \alpha \int_{{s_1}}^{{s_2}} {\left\| \Sigma^2_{i=1}\nabla_i\partial_i u\right\|_{L_x^2}^2} ds \lesssim \left(\left\| {\nabla u({s_1})} \right\|_{L_x^2}^2- \left\| {\nabla u({s_2})} \right\|_{L_x^2}^2 \right),\label{huhua4}\\
&\mbox{  }\mbox{  }\mbox{ }\int^T_0\int_{\Bbb R^2}|\nabla u|^4dydt\lesssim \|E(u(t);B_R(x))\|_{L^{\infty}_{t,x}(I\times\Bbb R^2)}\big(\int^T_0\int_{\Bbb R^2}|\nabla^2u|^2dydt +\frac{T}{R^2}E(u_0)\big)\label{huhua5};\\
&(iii)\mbox{ }E(u({s_2});{B_R}(x)) \le E(u({s_1});{B_{2R}}(x)) + \frac{{{C_3}({s_2} - {s_1})}}{{{R^2}}}E({u_0}), \label{huhua3}\\
&\mbox{  }\mbox{  }\mbox{  }\mbox{ }\mbox{ }\mbox{ }E(u({s_2});{B_{2R}}(x)) \ge E(u({s_1});{B_R}(x)) -(E(u(s_1))-E(u(s_2)))- \frac{{{C_3}({s_2} - {s_1})}}{{{R^2}}}E({u_0});\label{huhua1}\\
&(iv)\mbox{ }E(u(t)) {\rm{\mbox{ }is}} {\rm{\mbox{  }continuous \mbox{  }and \mbox{  }decreasing \mbox{  }with\mbox{  }respect\mbox{  } to}} \mbox{  }t;\label{huhua2}\\
&(v)\mbox{  }\exists \mbox{  }{\rm{ classical}} \mbox{  }{\rm{soltuion }}\mbox{  }{u_n}\mbox{  }{\rm{ with }}\mbox{  } {\left\| {{u_n}(0,x) - {u_0}(x)} \right\|_{{W^{1,2}}}} \to 0,\mbox{  }\mbox{  }{\partial _t}{u_n} \to {\partial _t}u\mbox{  }{\rm{ weakly}} \mbox{  }{\rm{in }}\mbox{  } L_{t,x}^2(I \times \Bbb R^2)\nonumber\\
&\mbox{  }\mbox{  }\mbox{  }{\rm{and }}\mbox{  }\left\| {u_n} - u \right\|_{C(I;L^2(\Bbb R^2))} \to 0,\mbox{  }\left\| D{u_n} - Du \right\|_{L^2(I;L^2(\Bbb R^2))} \to 0\nonumber.
\end{align}
Then for any initial data $u_0\in W^{1,2}$, there exists a $T>0$ such that (\ref{1}) admits a weak solution $u(t,x)\in \mathcal{H}([0,T]\times\Bbb R^2)$. And the weak solution is unique in $\mathcal{H}([0,T]\times\Bbb R^2)$.
\end{proposition}

In the following, we prove Theorem \ref{zaq}. First, by the method of induction on energy we obtain the boundedness of the space-time norm  $\|\nabla u\|_{L^4_{t,x}}$. This global $L^4_{t,x}$ norm has been explicitly used in Smith \cite{Sm} and the local version was initially used by Struwe \cite{S} in the heat flow case. Second, we rewrite (\ref{1}) under the Coulomb gauge. Furthermore, we give the estimates of the connection matrix $A_i$ by the intrinsic norm $\|\nabla^2u\|_2$  and $\|\nabla u\|_2$. Finally, the decay of the energy follows by applying the Strichartz estimates to the gauged equation.

\subsection{Rewrite the equation under the Gauge}
In this section, we present the gauged equation of (\ref{1}).
Assume that $u:[0,T]\times\Bbb R^2\to\mathcal{N}$ is a solution of (\ref{1}). Choose an orthonormal frame $\{e_1,...,e_{2l}\}$ for $u^*T\mathcal{N}$ with respect to $h$, and $e_{l+1}=J(e_1),...,e_{2l}=J(e_l)$. Let the Latin indices take values in $\{1,2,...,2l\}$, the Roman indices in $\{1,2\}$, and the Greek indices in $\{1,2,...,l\}$. We make the convention that $f^{\overline{\gamma}}=f^{\gamma+l}$, $f^{\overline{\gamma+l}}=f^{\gamma}$ for vector-valued functions $(f^1,...,f^{2l})^t$, and $e_{\overline{\gamma}}=e_{\gamma+l}$, $e_{\overline{\gamma+l}}=-e_{\gamma}$.
Expand $\nabla_{x,t}u$ in the frame $\{e_j\}$ as follows:
$${\partial _{x_i}}u = \sum\limits_{a = 1}^{2l} h({\partial _{x_i}}u,{e_a}){e_a} \equiv {\psi^a_i}{e_a},\mbox{  }\mbox{  }\mbox{  }\mbox{  }{\partial _t}u = \sum\limits_{a = 1}^{2l} h( {\partial _t}u,{e_a}){e_a} \equiv {b^a}{e_a}.$$
Since $J$ commutes with $\nabla_{x,t}$, rewriting (\ref{1}) by $\psi^a_i,b^a$ gives
\begin{align}\label{2}
b^ae_a=\sum\limits_{i=1}^{2}\alpha\big(\partial_{x_i}\psi^a_ie_a+\psi^a\nabla_{x_i}e_a\big)-\beta\big(\partial_{x_i}\psi^a_iJ(e_a)-\psi^a_i\nabla_{x_i}J(e_a)\big).
\end{align}
Denote the space of $\Bbb C^l$-valued field defined in $[0,T]\times\Bbb R^2$ by $\mathfrak{X}$, then $\nabla$ induces a covariant derivative on $\mathfrak{X}$ defined by
\begin{align*}
{D_i}{(v)^\gamma } &=\partial_{x_i}v^{\gamma}+ \big([{A_i}]_\theta ^\gamma  + \sqrt{-1}[A_i]_\theta ^{\bar \gamma }\big){v^\theta },\\
{D_t}{(v)^\gamma } &=\partial_{t}v^{\gamma}+ \big([{A_t}]_\theta ^\gamma  + \sqrt{-1}[A_t]_\theta ^{\bar \gamma }\big){v^\theta },
\end{align*}
where the corresponding connection coefficients matrices are given by
\begin{align*}
[{A_i }]_a^b &= \left\langle {{\nabla _{{x_i }}}{e_a},{e_b}} \right\rangle.
\end{align*}
Considering the complexification of $\psi_{i,t}$, $A_{i,t}$ defined by
\begin{align*}
[\widetilde{A}_t]^\gamma_\theta&=[{A_t}]_\theta ^\gamma  + \sqrt{-1}[A_t]_\theta ^{\bar \gamma },[\widetilde{A}_i]^\gamma_\theta=[{A_i}]_\theta ^\gamma  + \sqrt{-1}[A_i]_\theta ^{\bar \gamma }.\\
\phi_t&=(b^1+\sqrt{-1}b^{\overline{1}},...,b^{l}+\sqrt{-1}b^{\bar{l}})^t,\\ \phi_i&=(\psi^1_i+\sqrt{-1}\psi^{\overline{1}}_i,...,b^l+\sqrt{-1}\psi^{\bar{l}}_i)^t,
\end{align*}
then by $h\left\langle {JX,JY} \right\rangle  = h\left\langle {X,Y} \right\rangle$, for $X,Y\in T\mathcal{N}$, we can rewrite  (\ref{2}) as
\begin{align}\label{b.2}
\phi_t=\sum\limits_{i = 1}^2 \alpha D_i\phi_i-\sqrt{-1}\beta D_i\phi_i.
\end{align}
The following covariant curl-free identity and commutator identity are useful later
\begin{align}\label{b.3}
D_i\phi_j=D_j\phi_i, D_t\phi_i=D_i\phi_t, [D_i,D_j]v=\mathcal{R}(u)(\phi_i,\phi_j)v, [D_t,D_j]v=\mathcal{R}(u)(\phi_t,\phi_j)v,
\end{align}
where $\mathcal{R}(,)$ is a tensor with the pointwise estimate
\begin{align}\label{b.4}
|\mathcal{R}(u)(\phi_t,\phi_j)v|_{\Bbb C^n}\lesssim |\phi_t|_{\Bbb C^n}|\phi_j|_{\Bbb C^n}|v|_{\Bbb C^n}.
\end{align}
We use $|.|_{\Bbb C^n}$ here to emphasize that it is not the metric in $\mathcal{N}$. 
Applying (\ref{b.3}) to (\ref{b.2}), we obtain the equation for $\phi_j$
\begin{align*}
 {D_t}{\phi _j} &= {D_{{j}}}{\phi _t} = \sum\limits_{i = 1}^2 \alpha  {D_j}{D_i}{\phi _i} - \sqrt { - 1} \beta {D_j}{D_i}{\phi _i} \\
 &= \sum\limits_{i = 1}^2 {\alpha {D_i}{D_j}{\phi _i}}  - \sqrt { - 1} \beta {D_i}{D_j}{\phi _i} + \alpha\mathcal{R}({\phi _i},{\phi _j}){\phi _i} - \sqrt { - 1} \beta \mathcal{R}({\phi _i},{\phi _j}){\phi _i} \\
 &= \sum\limits_{i = 1}^2 {\alpha {D_i}{D_i}{\phi _j}}  - \sqrt { - 1} \beta {D_i}{D_i}{\phi _j} + \alpha\mathcal{R}({\phi _i},{\phi _j}){\phi _i} - \sqrt { - 1} \beta \mathcal{R}({\phi _i},{\phi _j}){\phi _i}.
 \end{align*}
This can be written as a Ginzburg-Landau type equation as follows
\begin{align}\label{d.61}
{\partial _t}{\phi _j} - z\Delta {\phi _j} = {\widetilde{A}_t}{\phi _j} + \sum\limits_{i = 1}^2 {z{\widetilde{A}_i}{\partial _i}{\phi _j}}  + z{\partial _i}{\widetilde{A}_i}{\phi _j} + z{\widetilde{A}_i}{\widetilde{A}_i}{\phi _j} + z\mathcal{R}({\phi _i},{\phi _j})\phi_i,
\end{align}
where $z=\alpha- \sqrt { - 1} \beta$.

If $\mathcal{N}$ is a Riemannian surface, we can choose the frame $\{e_1,e_2\}$ to be a Coulomb gauge, namely $\partial_i \widetilde{A}_i=0$, see for instance Nahmod, Shatah, Vega, Zeng \cite{NSV}. In this case, for $i\in\{1,2\}$, $[\widetilde{A}_i]=a_i$, $[\widetilde{A}_t]=a_t$ where $a_{i,t}$ is some pure-imaginary valued function defined on $[0,T]\times\Bbb R^2$. Moreover, (\ref{d.61}) simplifies to
\begin{align}\label{d.6}
\left\{ \begin{array}{l}
 {\partial _t}{\phi _j} - z\Delta {\phi _j} = {a_t}{\phi _j} + \sum\limits_{i = 1}^2 \big({z{a_i}{\partial _i}{\phi _j}}  + z{\partial _i}{a_i}{\phi _j} + z{a_i}{a_i}{\phi _j} + z{\cal R}({\phi _i},{\phi _j}){\phi _i}\big) \\
 \Delta {a_j} = {\partial _k}\left( {\mathbf{i}\kappa (u)\left\langle {{\phi _k},\mathbf{i}{\phi _j}} \right\rangle } \right) \\
 \Delta {a_t} =  - {\partial _k}\left( {\mathbf{i}\kappa (u)({\partial _j}\left\langle {{\phi _k},{\phi _j}} \right\rangle  - \frac{1}{2}{\partial _k}{{\left| {{\phi _j}} \right|}^2})} \right) \\
 \partial_ka_k=0\\
 \end{array} \right.
\end{align}
where $z=\alpha- \sqrt { - 1} \beta$, $\mathbf{i}=\sqrt{-1}$, $\left\langle {z,w} \right\rangle  = {\mathop{\rm Re}\nolimits} (z\bar w)$, $\kappa$ is the Gauss curvature.

The following lemma gives the bounds of the connection coefficient matrices by the covariant derivatives of $u$.
\begin{lemma}\label{July3}
If $\phi_{t,j}$, $a_{t,j}$ solves (\ref{d.6}), then for any $p\in(2, \infty)$, we have
\begin{align}
\|a_j\|_{L^p_x}&\lesssim \sum^2_{k=1}\|\phi_k\phi_j\|_{L^{p*}_x}\label{July1}\\
\|a_t\|_{L^p_x}&\lesssim \sum^2_{k=1}\||\nabla^2u||\nabla u|+|\nabla u|^2|a_k|\|_{L^{p*}_x},\label{July2}
\end{align}
where $\frac{1}{p*}+\frac{1}{2}=1+\frac{1}{p}$.
\end{lemma}
\begin{proof}
Since ${a_j} =  - {\partial _k}{( - \Delta )^{ - 1}}\left( {{\bf{i}}\kappa (u)\left\langle {{\phi _k},{\bf{i}}{\phi _j}} \right\rangle } \right)$, where $(-\Delta)^{-1}$ is expressed by the Newton potential, then $(\ref{July1})$ follows from weak Hausdorff-Young inequality. By the definition of $\phi_k$, we have
\begin{align*}
\partial_j\phi_k&=\partial_j(\left\langle {{\partial _k}u,{e_1}} \right\rangle  + \sqrt { - 1} \left\langle {{\partial _k}u,{e_{\bar 1}}} \right\rangle )\\
&=\left\langle {{\nabla _j}{\partial _k}u,{e_1}} \right\rangle  + \left\langle {{\partial _k}u,{\nabla _j}{e_1}} \right\rangle  + \sqrt { - 1} \left\langle {{\nabla _j}{\partial _k}u,{e_{\bar 1}}} \right\rangle  + \sqrt { - 1} \left\langle {{\partial _k}u,{\nabla _j}{e_{\bar 1}}} \right\rangle.
\end{align*}
Therefore, by the identities $\left\langle {{\nabla _j}{e_1},{e_1}} \right\rangle  = \left\langle {{\nabla _j}{e_{\bar 1}},{e_{\bar 1}}} \right\rangle  = 0$, $\left\langle {{\nabla _j}{e_1},{e_{\bar 1}}} \right\rangle  =  - \left\langle {{\nabla _j}{e_{\bar 1}},{e_1}} \right\rangle,$
we obtain
\begin{align}\label{July4}
|\partial_j\phi_k|\lesssim |\nabla^2u|+|\nabla u||a_j|.
\end{align}
Since ${a_t} = {\partial _k}{( - \Delta )^{ - 1}}\left( {{\bf{i}}\kappa (u)({\partial _j}\left\langle {{\phi _k},{\phi _j}} \right\rangle  - \frac{1}{2}{\partial _k}{{\left| {{\phi _j}} \right|}^2})} \right)$, (\ref{July4}) implies (\ref{July2}).
\end{proof}

The proof of the following Strichartz estimates is almost the same as the heat semigroup, thus we state it without proof.
\begin{lemma}\label{fghj}
Let $z$ be a complex number with $\mathfrak{Re}z>0$. Then for an admissible pair $(p,q)$ satisfying $\frac{1}{p}+\frac{1}{q}=\frac{1}{2}$, $2\le p,q\le\infty$, $p\neq 2$, and any pair $(r,s)$ satisfying $\frac{1}{r'}+\frac{1}{s'}=\frac{1}{2}$, $1\le r,s\le 2$, $r\neq 2$, we have
$$
\left\| {{e^{zt\Delta }}f} \right\|_{L_t^pL_x^q} \lesssim {\left\| f \right\|_{{L^2}}},\mbox{  }\mbox{  }
\left\| \int^t_{t_0}e^{z(t-\tau)\Delta}g(\tau,x)d\tau\right\|_{L^p_tL^q_x([t_0,t_1]\times \Bbb R^2)}\lesssim \|g\|_{L^r_tL^s_x([t_0,t_1]\times \Bbb R^2)}.
$$
\end{lemma}

In the rest of this section, we prove Theorem \ref{zaq} by assuming Theorem \ref{1}. The proof of Theorem \ref{1} is postponed to Section 3. We first remark that critical energy $E_*$ is always strictly positive for any compact target.
\begin{lemma}\label{cri}
For any compact K\"ahler manifold $\mathcal{N}$, the critical energy $E_*$ defined by (\ref{zas33}) is strictly positive, furthermore we have
$$
E_*\ge \frac{1}{2}\frac{1}{C^4_{1,2}R_{\mathcal{N}}},
$$
where $R_{\mathcal{N}}$ is the upper bound for the sectional curvature of $\mathcal{N}$, $C_{1,2}$ is the sharp constant for Gagliardo-Nirenberg inequality  $\|f\|_4\le C_{1,2}\|f\|_2^{\frac{1}{2}}\|\nabla f\|_2^{\frac{1}{2}}$.
\end{lemma}
\begin{remark}
We remark that Weinstein \cite{We} has proved $C_{1,2}$ is exactly achieved by the ground state of
$$
\Delta f-f+f^3=0,
$$
and $C_{1,2}={\left( {\frac{1}{{\pi \left( {1.86225....} \right)}}} \right)^{1/4}}.$
\end{remark}
The lower bound for $E_*$ given in Lemma \ref{cri} is not optimal. For instance, it is known that $E_*=4\pi$ if $\mathcal{N}$ is $\Bbb S^2$, and the bound obtained in Lemma \ref{cri} is ${\pi\times0.93112....}.$

\begin{proof}
If there is no harmonic map with finite energy, we have made the convention that $E_*=\infty$, thus it suffices to prove Lemma \ref{cri} when $E_*<\infty$.
Suppose that $u$ is a harmonic map from $\Bbb R^2$ to $\mathcal{N}$ satisfying
\begin{align}
\sum^2_{i=1}\nabla_i\partial_i u&=0, \label{okn}\\
0<\|\nabla u\|_{L^2_x}&<\infty. \label{okn1}
\end{align}
Integration by parts gives
\begin{align}\label{mnbv}
\int_{\Bbb R^2}|\nabla^2 u|^2dx\le R_{\mathcal{N}} \int_{\Bbb R^2} |\nabla u|^4dx +\int_{\Bbb R^2} |\sum^2_{i=1}\nabla_i\partial_i u|^2dx,
\end{align}
which combined with (\ref{okn}) yields that
\begin{align}\label{knvc}
\|\nabla^2 u\|^2_{L^2_x}\le R_{\mathcal{N}}\|\nabla u\|^4_{L^4_x}.
\end{align}
By Gagliardo-Nirenberg inequality, we have
\begin{align}\label{knvc21}
\|\nabla u\|_{L^4_x}\le C_{1,2} \|\nabla^2 u\|^{\frac{1}{2}}_{L^2_x} \|\nabla u\|^{\frac{1}{2}}_{L^2_x}.
\end{align}
Then (\ref{knvc21}), (\ref{knvc}) yield
\begin{align}\label{knvc1}
\|\nabla u\|_{L^4_x}\le  C_{1,2}R^{1/4}_{\mathcal{N}}\|\nabla u\|_{L^4_x} \|\nabla u\|^{\frac{1}{2}}_{L^2_x}.
\end{align}
Since $\|\nabla u\|_{L^2_x}> 0$ , we obtain
$$
\|\nabla u\|_{L^2_x}\ge \frac{1}{C^2_{1,2}R^{1/2}_{\mathcal{N}}}.
$$
\end{proof}

Theorem \ref{zaq} is proved by the method of energy induction due to Bourgain \cite{B}.
The classical line for the induction on energy argument involves three main ingredients: the scattering for small data; the existence of the critical elements; ruling out the critical elements.  The small data scattering lemma is given below. In the proof of the following lemma, we need to use some exponents, for the simplicity of the presentation, we introduce some notations.
For $2<p<\infty$, we define $p*$ by $\frac{1}{p*}=\frac{1}{p}+\frac{1}{2}.$ For $m,n\in [1,\infty)$, we define $(m,n)$ by $\frac{1}{(m,n)}=\frac{1}{m}-\frac{1}{n}$. The dual Strichartz exponent $\widehat{r}$ for $r\in(1,2)$ is define by $\frac{1}{\widehat{r}}=\frac{3}{2}-\frac{1}{r}$.

\begin{lemma}\label{zaq1}
Let $\varepsilon>0$ be sufficiently small. For any initial data $u_0\in W^{1,2}$ satisfying $\|\nabla u_0\|_{L^2_x}<\varepsilon$, (\ref{1}) has a unique global solution in $\mathcal{H}(\Bbb R^+\times\Bbb R^2)$, furthermore
we have
\begin{align}\label{hvc}
\int^{\infty}_{0}\int_{\Bbb R^2} |\nabla u(t,x)|^4dxdt\le C,
\end{align}
for some $C>0$.
\end{lemma}
\begin{proof}
Let $\varepsilon^2<2E_*$, the global well-posedness is a corollary of Proposition 3.4. In fact, if $u$ blows up at some finite time $T>0$, then by Proposition 3.4, there exists a non-trivial harmonic map $U(x)$ which is a weak limit of the rescaling and translation of $u(t_n,x)$. Then we have
$$E(U)\le E(u_0)<E_*.$$
This contradicts with the definition of $E_*$. Hence, $u_0$ evolves to a unique global solution in $\mathcal{H}$ defined in Proposition 2.1. Then we prove (\ref{hvc}) by a bootstrap argument. Define
$$
\mathcal{A}=\{T>0, \|\nabla u(t,x)\|_{L^4_{t,x}([0,T]\times \Bbb R^2)}\le C^*\varepsilon\},
$$
where $C^*>0$ will be determined later. The non-empty and closed-ness of $\mathcal{A}$ follows from (\ref{huhua5}) which implies
$$
 \|\nabla u(t,x)\|_{L^4([s,s']\times \Bbb R^2)}\lesssim  E(u_0)\|\nabla^2 u(t,x)\|_{L^2([s,s']\times \Bbb R^2)},
$$
and the fact that $u\in \mathcal{H}$. It remains to prove the openness of $\mathcal{A}$.
Assume that $T\in \mathcal{A}$, it suffices to show
\begin{align}\label{hjnb}
\|\nabla u(t,x)\|_{L^4([0,T]\times \Bbb R^2)}\le \frac{1}{100}C^*\varepsilon.
\end{align}
(\ref{mnbv}), (\ref{huhua4}) and the bootstrap assumption $T\in \mathcal{A}$ imply
\begin{align}\label{zsa}
\|\nabla^2u\|^2_{L^2_{t,x}([0,T]\times \Bbb R^2)}\lesssim (C^*\varepsilon)^4+\varepsilon^2.
\end{align}
Consider (\ref{d.6}), Strichartz estimates in Lemma \ref{fghj} yield for some $n=2^-$
\begin{align}\label{strigh}
\|\phi_j\|_{L^4_tL^4_x}\lesssim \|\phi_j(0)\|_{L^2}+ \|{a_t}{\phi _j}\|_{L^{\widehat{n}}_tL^n_x}+ \sum\limits_{i = 1}^2 \|{a_i}{\partial _i}{\phi _j}\|_{L^{\widehat{n}}_tL^{n}_x}  + \|{a_i}{a_i}{\phi _j}\|_{L^{\widehat{n}}_tL^{n}_x}+ \|{\phi _i}{\phi _j}{\phi _i}\|_{L^{1}_tL^2_x},
\end{align}
where the integration domains of the norms $L^p_tL^q_x$ are  $[0,T]\times\Bbb R^2$.
First, we bound  $\|{a_t}{\phi _j}\|_{L^{\widehat{n}}_tL^n_x}$. H\"older inequality, Lemma \ref{July3} show for $p\in (2,\infty)$
\begin{align*}
 {\left\| {{a_t}{\phi _j}} \right\|_{L_x^n}} &\le {\left\| {{a_t}} \right\|_{L_x^p}}{\left\| {{\phi _j}} \right\|_{L_x^{(n,p)}}} \\
 &\le {\left\| {\left| {{\nabla ^2}u} \right|\left| {\nabla u} \right|} \right\|_{L_x^{p * }}}{\left\| {\nabla u} \right\|_{L_x^{(n,p)}}} + {\left\| {{{\left| {\nabla u} \right|}^2}\left| {{a_j}} \right|} \right\|_{L_x^{p * }}}{\left\| {\nabla u} \right\|_{L_x^{(n,p)}}} \\
 &\le {\left\| {{\nabla ^2}u} \right\|_{L_x^2}}{\left\| { {\nabla u} } \right\|_{L_x^{(p * ,2)}}}{\left\| {\nabla u} \right\|_{L_x^{(n,p)}}} + \left\| {\nabla u} \right\|_{L_x^4}^2{\left\| {{a_j}} \right\|_{L_x^{(p * ,2)}}}{\left\| {\nabla u} \right\|_{L_x^{(n,p)}}} \\
 &\le {\left\| {{\nabla ^2}u} \right\|_{L_x^2}}{\left\| { {\nabla u}} \right\|_{L_x^{(p * ,2)}}}{\left\| {\nabla u} \right\|_{L_x^{(n,p)}}} + \left\| {\nabla u} \right\|_{L_x^4}^2{\left\| {{\phi _j}{\phi _k}} \right\|_{L_x^{(p * ,2) * }}}{\left\| {\nabla u} \right\|_{L_x^{(n,p)}}} \\
 &\triangleq I + II.
\end{align*}
By Gagliardo-Nirenberg inequality, we have
\begin{align}\label{uuhk}
I \le {\left\| {{\nabla ^2}u} \right\|_{L_x^2}}\left\| {\nabla u} \right\|_{L_x^{2}}^{{\theta _p}}\left\| {{\nabla ^2}u} \right\|_{L_x^{2}}^{1 - {\theta _p}}\left\| {\nabla u} \right\|_{L_x^2}^{\frac{2}{{(n,p)}}}\left\| {{\nabla ^2}u} \right\|_{L_x^2}^{1 - \frac{2}{{(n,p)}}},
\end{align}
where ${\theta _p} = \frac{2}{{p * }} - 1.$
Since we have
$\frac{1}{{(p * ,2) * }} = \frac{1}{p} + \frac{1}{2}$, then Gagliardo-Nirenberg inequality implies
\begin{align}\label{pojh}
II \le {\left\| {\nabla u} \right\|_{L_x^2}}{\left\| {{\nabla ^2}u} \right\|_{L_x^2}}\left\| {\nabla u} \right\|_{L_x^2}^{\frac{2}{p} + 1}\left\| {{\nabla ^2}u} \right\|_{L_x^2}^{1 - \frac{2}{p}}\left\| {\nabla u} \right\|_{L_x^2}^{\frac{2}{{(n,p)}}}\left\| {{\nabla ^2}u} \right\|_{L_x^2}^{1 - \frac{2}{{(n,p)}}}.
\end{align}
Therefore (\ref{uuhk}), (\ref{pojh}) give the bound
\begin{align*}
{\left\| {{a_t}{\phi _j}} \right\|_{L_t^{\widehat{n}}L_x^n([0,T] \times {R^2})}} \le {\left( {\int_0^T {\left\| {{\nabla ^2}u} \right\|_{L_x^2}^{\left( {3 - \frac{2}{p} - \frac{2}{{(n,p)}}} \right)\widehat{n}}ds} } \right)^{1/{\widehat{n}}}}.
\end{align*}
Since $\left( {3 - \frac{2}{p} - \frac{2}{{(n,p)}}} \right)\widehat{n} = 2$, (\ref{zsa}) yields the acceptable bound for $a_t\phi_j$
\begin{align}\label{0110}
{\left\| {{a_t}{\phi _j}} \right\|_{L_t^{\hat n}L_x^n([0,T] \times {R^2})}} \lesssim {[{\varepsilon ^2} + {({C^*}{\varepsilon ^2})^4}]^{\frac{1}{{\hat n}}}}.
\end{align}
Second, we bound $a_i\partial_i\phi_j$. H\"older inequality, Lemma \ref{July3} and Gagliardo-Nirenberg inequality give
\begin{align}
 {\left\| {{a_i}{\partial _i}{\phi _j}} \right\|_{L_x^n}} &\lesssim {\left\| {{a_i}{a_j}\left| {\nabla u} \right|} \right\|_{L_x^n}} + {\left\| {{a_i}\left| {{\nabla ^2}u} \right|} \right\|_{L_x^n}} \label{ujkl}\\
 &\lesssim \left\| {{a_j}} \right\|_{L_x^{2p}}^2{\left\| {\nabla u} \right\|_{L_x^{(n,p)}}} + {\left\| {{a_i}} \right\|_{L_x^{(n,2)}}}{\left\| {{\nabla ^2}u} \right\|_{L_x^2}} \nonumber \\
 &\lesssim \left\| {{\phi _j}{\phi _k}} \right\|_{L_x^{\left( {2p} \right) * }}^2\left\| {\nabla u} \right\|_{L_x^2}^{\frac{2}{{(n,p)}}}\left\| {{\nabla ^2}u} \right\|_{L_x^2}^{1 - \frac{2}{{(n,p)}}} + \left\| {{\phi _j}{\phi _k}} \right\|_{L_x^{\left( {n,2} \right) * }}^2{\left\| {{\nabla ^2}u} \right\|_{L_x^2}}.\nonumber
 \end{align}
Since we have $\frac{1}{2}\frac{1}{{\left( {2p} \right) * }} = \frac{1}{{4p}} + \frac{1}{4}$, $\frac{1}{2}\frac{1}{{\left( {n,2} \right) * }} = \frac{1}{{2n}}$,
Gagliardo-Nirenberg inequality gives
 \begin{align*}
 \left\| {{\phi _j}{\phi _k}} \right\|_{L_x^{\left( {2p} \right) * }}^2 &\lesssim \left\| {\nabla u} \right\|_{L_x^2}^{2 + \frac{2}{p}}\left\| {{\nabla ^2}u} \right\|_{L_x^2}^{2 - \frac{2}{p}} \\
 \left\| {{\phi _j}{\phi _k}} \right\|_{L_x^{\left( {n,2} \right) * }}^2 &\lesssim \left\| {\nabla u} \right\|_{L_x^2}^{\frac{2}{n}}\left\| {{\nabla ^2}u} \right\|_{L_x^2}^{2 - \frac{2}{n}},
\end{align*}
Hence we obtain
$${\left\| {{a_i}{\partial _i}{\phi _j}} \right\|_{L_t^{\hat n}L_x^n([0,T] \times {R^2})}} \lesssim {\left( {\int_0^T {\left\| {{\nabla ^2}u} \right\|_{L_x^2}^{\left( {3 - \frac{2}{n}} \right)\widehat{n}}ds} } \right)^{\frac{1}{{\hat n}}}}.
$$
Then we deduce the acceptable bound for $a_i\partial_i \phi_j$ from ${\left( {3 - \frac{2}{n}} \right)\widehat{n}}=2$ and (\ref{zsa})
\begin{align}\label{00110}
{\left\| {{a_i}{\partial _i}{\phi _j}} \right\|_{L_t^{\hat n}L_x^n([0,T] \times {R^2})}} \lesssim {[{\varepsilon ^2} + {({C^*}{\varepsilon ^2})^4}]^{\frac{1}{{\hat n}}}}.
\end{align}
Third, we notice that the term $a_ia_k\phi_j$ has appeared in (\ref{ujkl}), thus we have the following bound for $a_ia_k\phi_j$
\begin{align}\label{01010}
{\left\| {a_ia_k\phi _j} \right\|_{L_t^{\hat n}L_x^n([0,T] \times {R^2})}} \lesssim {[{\varepsilon ^2} + {({C^*}{\varepsilon ^2})^4}]^{\frac{1}{{\hat n}}}}.
\end{align}
Finally, we bound $O(\phi_x^3)$. Again by Gagliardo-Nirenberg inequality, we have
$$\left\| \phi_x  \right\|_{L_x^6}^3 \le {\left\| {\nabla u} \right\|_{L_x^2}}\left\| {{\nabla ^2}u} \right\|_{L_x^2}^2.$$
Thus (\ref{zsa}) implies
\begin{align}\label{19999}
\left\| O(\phi_x)^3 \right\|_{L^1_tL^2_x}\lesssim \varepsilon^2+(C^*\varepsilon)^4.
\end{align}
We conclude from (\ref{strigh}), (\ref{0110}), (\ref{00110}), (\ref{01010}), (\ref{19999}) that
\begin{align*}
\|\phi_j\|_{L^{\infty}_tL_x^2}\lesssim \varepsilon+ \varepsilon^2+(C^*\varepsilon)^4+(\varepsilon^2+(C^*\varepsilon)^4)^{\frac{1}{\widehat{n}}}.
\end{align*}
Then first choosing $C^*$ sufficiently large, then taking $\varepsilon$ sufficiently small, we obtain (\ref{hjnb}).   Thus Lemma \ref{zaq1} follows.
\end{proof}

Now, we can prove the ``scattering norm" $\|\nabla u\|_{L^4_{t,x}([0,\infty)\times\Bbb R^2)}$ is finite for all $u_0$ with the energy below $E_*$.
\begin{lemma}\label{pohv}
For any initial data $u_0\in W^{1,2}$ satisfying $E(u_0)<E_*$, (\ref{1}) has a global unique solution in $\mathcal{H}(\Bbb R^+\times\Bbb R^2)$, furthermore
we have
\begin{align}
\int^{\infty}_{0}\int_{\Bbb R^2} |\nabla u(t,x)|^4dxdt\le C,
\end{align}
for some $C>0$.
\end{lemma}
\begin{proof}
We assume $E_*<\infty$ below, the case $E_*=\infty$ can be proved with some modifications.
Define the threshold energy $E^\#$ for the scattering by
\begin{align*}
 {E^\# } = \sup \{ &E:{\rm{if }}\mbox{  }E({u_0}) < E,{\rm{then }} \mbox{  }\|\nabla u(t,x)\|_{L_{t,x}^4([0,\infty ) \times {\Bbb R^2})} < C(E)  \\
 &{\rm{ for\mbox{  } some }}\mbox{  }C(E)\mbox{  }{\rm{ depending \mbox{  } only \mbox{  }on}} \mbox{ }E \}.
\end{align*}
It is clear that $E^\#\le E_*$ because any non-trivial harmonic map $U(x)$ solves (\ref{1}) but we have
${\left\| {U(x)} \right\|_{L_{t,x}^4([0,\infty ) \times {R^2})}}=\infty$. Moreover, Lemma \ref{zaq1} shows $E^{\#}>0$. We prove this lemma by a contradiction argument. Suppose that $E^\#<E_*-\delta$, for some $\delta>0$, then we obtain a sequence of solutions of (\ref{1}) which satisfy
\begin{align}
E^{\#}-\frac{1}{n}<\frac{1}{2}\|\nabla u_n(0,x)\|^2_{L^2_x}&<E^{\#},\\
\mathop {\lim }\limits_{n \to \infty } {\left\| {\nabla{u_n}(t,x)} \right\|_{L_{t,x}^4([0,\infty ) \times {\Bbb R^2})}}&=\infty.\label{bv}
\end{align}
Let $\mu$ be a fixed positive constant. By (\ref{bv}), there exists a time sequence $\{t_n\}$ such that
\begin{align}\label{hcx}
{\left\| {\nabla{u_n}(t,x)} \right\|_{L_{t,x}^4([0,{t_n}] \times {\Bbb R^2})}} = \mu.
\end{align}
We claim that there exists a subsequence of $\{t_n\}$ such that
\begin{align}\label{uhv}
{E^\# } - \frac{1}{k} < \frac{1}{2}\left\| {\nabla {u_{{n_k}}}({t_{{n_k}}},x)} \right\|_{L_x^2}^2 < {E^\# }.
\end{align}
Indeed, if the claim fails, then there exits some constant $\varrho>0$ such that $E(u_{n}(t_{n}))<E^{\#}-\varrho$. Thus the solution to (\ref{1}) with initial data $u_{n}(t_{n},x)$ has a finite $L^4_{t,x}$ norm, then (\ref{hcx}) yields
$$
\|\nabla u_{n}(t,x)\|_{L^4_{t,x}([0,\infty)\times\Bbb R^2)}\le C(E^{\#}-\varrho)+\mu.
$$
This contradicts with (\ref{bv}). By the scaling invariance, we can assume $t_n=1$, then we conclude that for some solution sequence $\{u_n\}$  of (\ref{1})
\begin{align}
 {E^\# }- \frac{1}{n}< \frac{1}{2}\left\| {\nabla {u_n}(0,x)} \right\|_{L_x^2}^2 &< {E^\# } \label{1xz}\\
 {E^\# } - \frac{1}{n} < \frac{1}{2}\left\| {\nabla {u_n}(1,x)} \right\|_{L_x^2}^2 &< {E^\# } \label{2xz}\\
 {\left\| {\nabla {u_n}(1,x)} \right\|_{L_t^4L_x^4([0,1] \times {R^2})}} &= \mu.
\end{align}
From the energy identity (\ref{huhua4}), we have
$$
E(u_n(1,x))-E(u_n(0,x))\le -\alpha\int^1_0\|\Sigma_{i=1}^2\nabla_i\partial_i u\|^2_{L^2_x}ds.
$$
Then (\ref{mnbv}) implies
$$
E(u_n(1,x))-E(u_n(0,x))+\alpha\int^1_0\|\nabla^2 u\|^2_{L^2_x}ds\lesssim \int^1_0\|\nabla u\|^4_{L^4_x}ds.
$$
Hence by (\ref{1xz}) and (\ref{2xz}), for $n$ sufficiently large
\begin{align}\label{3zx}
\int^1_0\|\nabla^2 u_n\|^2_{L^2_x}ds\lesssim \mu^4.
\end{align}
On the other hand, (\ref{huhua5}) yields for any $R>0$,
\begin{align}\label{4zx}
\mu^4=\|u_n\|^4_{L^4_{t,x}([0,1]\times\Bbb R^2)}\lesssim \|E(u_n;B_R(x))\|_{L^{\infty}([0,1]\times\Bbb R^2)}\left(\|\nabla^2u_n\|^2_{L^2_{t,x}([0,1]\times\Bbb R^2)}+\frac{1}{R^2}E(u_n(0))\right).
\end{align}
Hence we have from (\ref{3zx}) and (\ref{4zx}) that
\begin{align}\label{pknbv}
\mu^4\lesssim \|E(u_n;B_R(x))\|_{L^{\infty}_{t,x}([0,1]\times\Bbb R^2)}(\mu^4+\frac{E_*}{R^2}).
\end{align}
Assume $R>C_1\frac{E_*^{\frac{1}{2}}}{\mu^2}$ for some sufficiently large universal constant $C_1$, (\ref{pknbv}) yields
\begin{align}\label{zx9}
4c_2\le \|E(u_n;B_R(x))\|_{L^{\infty}_{t,x}([0,1]\times\Bbb R^2)},
\end{align}
where $c_2$ is some small universal constant.
Thus we can choose $x_n\in \Bbb R^2$, $s_n\in[0,1]$ such that
\begin{align}\label{zx72}
4c_2\le E(u_n(s_n);B_R(x_n)).
\end{align}
We claim that $s_n$ can be chosen such that $s_n\ge \frac{c_2R^2}{10^mCE_*}$, for some $m$ sufficiently large, and
\begin{align}\label{zx7}
c_2\le E(u_n(s_n);B_{2R}(x_n)).
\end{align}
In order to prove (\ref{zx7}), consider two subcases:
\begin{align*}
&(a)\mathop {\lim \sup}\limits_{n \to \infty }  {s_n} < 1,\\
&(b)\mathop {\lim \sup}\limits_{n \to \infty }  {s_n} =1.
\end{align*}
For the case $(a)$ , without loss of generality, we can assume $s_n\le 1-\sigma$ for some $\sigma>0$.
Meanwhile (\ref{huhua1}) implies for all $\lambda R^2<\sigma$
\begin{align}\label{h9vc}
E(u_n(s_n+\lambda R^2); B_{2R}(x_n))\ge E(u_n(s_n); B_R(x_n))-C\big(E(u_n(s_n))-E(u_n(s_n+\lambda R^2))\big)-C\lambda E_*.
\end{align}
By the decreasing of energy (\ref{huhua2}) and  (\ref{1xz}), (\ref{2xz}), we obtain
$$
\mathop {\lim }\limits_{n \to \infty } E(u_n(s_n))-E(u_n(s_n+\lambda R^2))=0.
$$
Therefore (\ref{h9vc}) implies that for sufficiently large $n$ and $\lambda\in (\frac{c_2}{10^mCE_*},\frac{c_2}{10^{m-1}CE_*})$, we have
\begin{align}\label{x2sdf}
c_2\le E(u_n(s_n+\lambda R^2);B_{2R}(x_n)).
\end{align}
Thus without loss of generality, in the case $(a)$ we can assume $s_n\ge \frac{c_2R^2}{10^mCE_*}$, where $m\in \Bbb Z^+$ is sufficiently large to guarantee $\frac{R^2c_2}{10^mCE_*}<1$. In the case $(b)$ , it is obvious that we can also assume  $s_n\ge \frac{c_2R^2}{10^mCE_*}$.
Applying (\ref{huhua3}), for $s\in[0,s_n]$, we get
\begin{align}\label{zx5}
E(u_n(s_n);B_r(x_n))\le E(u_n(s);B_{2r}(x_n))+\frac{C_3(s_n-s)}{r^2}E(u_n(s)).
\end{align}
Let $r^2=\max(\frac{C_3E_*}{c_2},R^2)$, then (\ref{zx7}) gives
\begin{align}\label{zx5}
E(u_n(s);B_{2r}(x_n))\ge \frac{1}{2}c_2,
\end{align}
for all $s\in[0,s_n]$.
Let $\widetilde{u}_n(s,x)=u_n(sr^2,x_n+xr)$, then $\{\widetilde{u}_n\}$ defined on $I\times\Bbb R^2$ with $I\triangleq[0,\frac{c_2}{10^mCE_*}]$ satisfy
\begin{align}
E(\widetilde{u}_n(s);B_1(0))&\ge \frac{1}{2}c_2.\label{2hv}\\
E(\widetilde{u}_n)&\le E^{\#}.\label{1hvcx}\\
\mathop {\lim }\limits_{n \to \infty }\|\Sigma^2_{i=1}\nabla_i\partial_i \widetilde{u}_n(s)\|_{L^2_{t,x}([0,I]\times\Bbb R^2)}&=0.\label{zxd2}
\end{align}
Notice that (\ref{zxd2}) follows from the energy identity (\ref{huhua4})  and $E(u_n(0))\to E(u_n(1))$ as $n\to \infty$.
Following the arguments in  Theorem 4.3 of Struwe \cite{S}, we have from (\ref{2hv}), (\ref{1hvcx}) and (\ref{zxd2}) that there exists a non-trivial harmonic map $U:\Bbb R^2\to \mathcal{N}$ such that $E(U)\le E^{\#}<E_*-\delta$. This contradicts with the definition of $E_*$. Therefore, $E^{\#}=E_*$ thus the $E_*<\infty$ case in Lemma \ref{pohv} has been verified.
For the case $E_*=\infty$, if Lemma \ref{pohv} fails, then we have $E^\#<\infty$. Then all the arguments above work with the upper bound $E_*$ in the estimates replaced by $E^\#$. Hence in the case $E_*=\infty$, Lemma \ref{pohv} follows as well.
\end{proof}

Now we are ready to prove Theorem \ref{zaq}.
\begin{proposition}
Let $(\mathcal{N},h,J)$ be a compact Riemannian surface, $\alpha>0$, $\beta\in \Bbb R$. The LL flow with $u_0\in W^{1,2}(\Bbb R;\mathcal{N})$ satisfying $E(u_0)<E_*$ admits a global unique solution $u\in \mathcal{H}([0,\infty)\times \Bbb R^2)$. Moreover, $u(t,x)$ converges to a constant map as $t\to\infty$ in the energy space, namely
$$
\mathop {\lim }\limits_{t \to \infty } E(u(t)) = 0.
$$
\end{proposition}
\begin{proof}
The global existence of $u$ and (\ref{huhua4}) imply
\begin{align*}
\int^{\infty}_0\|\sum^2_{i=1}\nabla_i\partial_i u\|^2_{L^2_x}ds<\infty.
\end{align*}
Therefore we infer from Lemma \ref{pohv} and (\ref{mnbv}) that
\begin{align*}
\int^{\infty}_0\int_{\Bbb R^2} |\nabla^2 u|^2dxdt\lesssim 1.
\end{align*}
For any $\varepsilon>0$, let $T>0$ be a sufficiently large constant such that
\begin{align}\label{12qq}
\int^{\infty}_T\int_{\Bbb R^2} |\nabla^2 u|^2dxdt\le \varepsilon^2.
\end{align}
Consider (\ref{d.6}), Strichartz estimates in Lemma \ref{fghj} yield for some $n=2^-$
\begin{align}\label{stri}
\|\phi_j(t)\|_{L^2_x}\lesssim \|e^{z\Delta(t-T)}\phi_j(T)\|_{L^2_x}+\|{a_t}{\phi _j}\|_{L^{\widehat{n}}_tL^n_x}+ \sum\limits_{i = 1}^2 \|{a_i}{\partial _i}{\phi _j}\|_{L^{\widehat{n}}_tL^{n}_x}  + \|{a_i}{a_i}{\phi _j}\|_{L^{\widehat{n}}_tL^{n}_x}+ \|{\phi _i}{\phi _j}{\phi _i}\|_{L^{1}_tL^2_x},
\end{align}
where the integration domains of the norms $L^p_tL^q_x$ are  $[T,t)\times\Bbb R^2$. Then the same arguments as Lemma \ref{zaq1} show
\begin{align}\label{020110}
\left\| {{a_t}{\phi _j}} \right\|_{L_t^{\hat n}L_x^n([T,t] \times {R^2})}+
\left\| {{a_i}{\partial_i\phi _j}} \right\|_{L_t^{\hat n}L_x^n([T,t] \times {R^2})}+
&\left\| {a_ia_k\phi _j} \right\|_{L_t^{\hat n}L_x^n([T,t] \times {R^2})}
+\left\| {\phi_i\phi_k\phi _j} \right\|_{L_t^1L_x^2([T,t] \times {R^2})}\nonumber\\
&\lesssim \int_T^\infty  \|\nabla ^2 u\|_{L_x^2}^2ds.
\end{align}
We conclude from (\ref{stri}), (\ref{020110}), (\ref{12qq}) that
\begin{align}\label{3056}
\|\phi_j(t)\|_{L_x^2}\lesssim \varepsilon^2+\|e^{z\Delta(t-T)}\phi_j(T)\|_{L^2_x}.
\end{align}
For this fixed $T$, let $t\to \infty$, by a standard density argument, we have
$\mathop {\lim }\limits_{t \to \infty } {\left\| {{e^{z(t - T)\Delta }}{\phi _j}(T)} \right\|_{L_x^2}} = 0$. Therefore for sufficiently large $t$ we have from (\ref{3056}) that
\begin{align*}
\|\phi_j(t)\|_{L_x^2}\lesssim 2\varepsilon^2.
\end{align*}
Then Theorem \ref{zaq} follows immediately.
\end{proof}

\section{Well-posedness and bubbling theorem}
In this section, we will prove the global existence of weak solutions to (\ref{1}) and establish a bubbling theorem. Although the method is an analogy to the case of the heat flow, we need to develop some cancelation of the high derivative terms due to the appearance of the complex structure term to close the energy estimates. The other difference is that $u$ is defined on a non-compact manifold, more efforts should be paid to apply compactness arguments.

First, we give the extrinsic formulation of (\ref{1}). Suppose that $\iota:\mathcal{N}\to \Bbb R^{m}$ is a fixed isometric embedding.
Let $\delta>0$ be  a chosen sufficiently small constant so that on the $\delta$-tubular neighborhood $\iota(N)_¦Ä\subset\Bbb R^m$, the
nearest point projection map
$$\Pi: \iota(\mathcal{N})_\delta\to  \iota(\mathcal{N})$$
is a smooth map. Note that $P(y)=d\Pi(y):\Bbb R^m\to T_y{\mathcal{N}}$, $y\in\mathcal{N}$, is an orthogonal projection map, and
$$
A(y)=\nabla P(y):T_y\mathcal{N}\otimes T_y\mathcal{N}\to (T_y\mathcal{N})^{\perp}, \mbox{  }y\in \mathcal{N},
$$
is the second fundamental form of $\mathcal{N}\subset \Bbb R^m$.

\begin{definition} If $v=\iota \circ u$, then the ambient form of the Schr\"odinger vector field
$J(u)\tau(u)$, is given by the vector field $F_v$ with
\begin{align}\label{FV}
F_v\triangleq d\iota {|_{{\iota ^{ - 1}}(\Pi (v(x))}}J\left( {{\iota ^{ - 1}}\Pi (v(x))} \right){(d\iota )^{ - 1}}{|_{\Pi (v(x)}}d\Pi {|_{v(x)}}(\Delta v).
\end{align}
\end{definition}
Notice that $F_v$ is defined for maps $v: R^m \to w(\mathcal{N})_\delta$ whose image do not necessarily
lie on $\mathcal{N}$. Moreover we remark that $F_v$ defined by (\ref{FV}) can be written in the following explicit form
\begin{align}\label{FV1}
F_v=B_1(v)(\Delta v)+B_2(v)\nabla v\ast\nabla v,
\end{align}
where $B_1,B_2$ are smooth bounded matrix-valued functions, $\nabla v\ast\nabla v$ denotes the quadratic terms of $\nabla v$.
Thus the extrinsic form of (\ref{1}) is given by
\begin{align}\label{FV3}
\partial_t v=M(v)[d\Pi|_{v(x)}(\Delta v)\big],
\end{align}
where $M(v) = {\left( {{\gamma _1} - {\gamma _2}\left( {d\iota {|_{{\iota ^{ - 1}}(\Pi (v(x))}}J\left( {{\iota ^{ - 1}}\Pi (v(x))} \right){{(d\iota )}^{ - 1}}{|_{\Pi (v(x))}}d\Pi {|_{\Pi (v(x))}}} \right)} \right)^{ - 1}}$,
$\gamma_1=\frac{\alpha}{\alpha^2+\beta^2}$, $\gamma_2=\frac{\beta}{\alpha^2+\beta^2}$. The existence of the inverse in $M(v)$ will be verified
in Lemma \ref{par} below.
When $\alpha>0$, (\ref{1}) is essentially a quasilinear parabolic system, which can be explained by the following lemma. In the $\Bbb S^2$ target case, $M(v)$ can be explicitly written down for instance \cite{H,ZGT}.

\begin{lemma}\label{par}
Suppose that $u:\Bbb R^2\times [0,T]\to \mathcal{N}$, $v=\iota\circ u$. Let $\alpha>0$, $V:\Bbb R^2\times[0,T]\to \Bbb R^m$ be a vector field, viewing $V(x,t)$ as an element of $T_{v(x,t)} \Bbb R^m$, then we have
\begin{align}\label{FV4}
\alpha|V(x,t)|^2\le V(x,t)^TM(v)V(x,t)\le \frac{1}{\gamma_1}|V(x,t)|^2.
\end{align}
\end{lemma}
\begin{proof}
Since $u(t,x)\in \mathcal{N}$, we have $\Pi v(x,t) = v(x,t)$, and $d\Pi|_{v(x,t)}$ is an orthogonal projection.
Define $\Phi\triangleq d\iota {|_{{\iota ^{ - 1}}(v(x))}}J\left( {{\iota ^{ - 1}}(v(x))} \right){(d\iota )^{ - 1}}{|_{v(x)}}d\Pi {|_{v(x)}}
$. First we show  ${\gamma _1} - {\gamma _2}\Phi$ is invertible. It suffices to prove all the eigenvalues of $\Phi$ do not vanish. Fixed $(x,t)\in\Bbb R^2\times [0,T]$, suppose that $\xi(x,t)$ is an eigenfunction of $\Phi$, namely for some $\lambda(x,t)\in \Bbb C$
\begin{align}\label{FV5}
{\gamma _1}\xi  - {\gamma _2}(d\iota {|_{{\iota ^{ - 1}}(v(x,t)}}J\left( {{\iota ^{ - 1}}v(x,t)} \right){(d\iota )^{ - 1}}{|_{v(x,t)}})d\Pi {|_{ v(x,t)}}\xi  = \lambda \xi.
\end{align}
Define the orthogonal decomposition of $\xi$ by
$${\xi _1} = d\Pi |_{v(x,t)}\xi ,\mbox{  }\mbox{  }{\xi _2} = \xi - d\Pi |_{v(x,t)}\xi.
$$
Taking the inner product with $\xi_1$ on both sides of (\ref{FV5}) yields
\begin{align}
{\gamma _1}{\left| {{\xi _1}} \right|^2} - {\gamma _2}\left\langle {d\iota {|_{{\iota ^{ - 1}}(v(x,t)}}J\left( {{\iota ^{ - 1}}v(x,t)} \right){{(d\iota )}^{ - 1}}{|_{v(x,t)}}d\Pi {|_{v(x,t)}}\xi ,d\Pi {|_{v(x,t)}}\xi } \right\rangle  = \lambda {\left| {{\xi _1}} \right|^2}.
\end{align}
Since $\iota$ is an isometric embedding, $(JX,X)=0$ for $X\in T\mathcal{N}$, we have
\begin{align*}
& \left\langle {d\iota {|_{{\iota ^{ - 1}}v(x,t)}}J\left( {{\iota ^{ - 1}}v(x,t)} \right){{(d\iota )}^{ - 1}}{|_{v(x,t)}}d\Pi {|_{v(x,t)}}\xi ,d\Pi {|_{v(x,t)}}\xi } \right\rangle  \\
&= \left\langle {J\left( {{\iota ^{ - 1}}v(x)} \right){{(d\iota )}^{ - 1}}{|_{v(x,t)}}d\Pi {|_{v(x,t)}}\xi ,{{(d\iota )}^{ - 1}}{|_{v(x,t)}}d\Pi {|_{v(x,t)}}\xi } \right\rangle=0.
\end{align*}
Thus if $\xi_1\neq0$, then $\lambda={\gamma _1}>0$. If $\xi_1=0$, then taking the inner product with $\xi_2$ on both sides of (\ref{FV5}) yields
$${\gamma _1}{\left| {{\xi _2}} \right|^2}= \lambda{\left| {{\xi _2}} \right|^2}.$$
Since in this case $\xi=\xi_2\neq0$, again we have $\lambda={\gamma _1}>0.$ Hence $\Phi$ is invertible. We use the following matrix norm induced by the Euclidean metric in $\Bbb R^m$:
$$\left\| A \right\| = \max \left\{ {\rho :{\rho ^2}{\rm{\mbox{  }is\mbox{  } an\mbox{  } eigenvalue\mbox{  } of}}\mbox{  }{A^*}A} \right\}
$$
Since $d\Pi|_{v(x,t)}$ is an orthogonal projection to $T_{v(x,t)}\mathcal{N}$ and  $\iota$ is an isometry embedding, $(JX,Y)=-(X,JY)$, we have
\begin{align*}
 &\left\langle {d\iota {|_{{\iota ^{ - 1}}(v(x,t)}}J\left( {{\iota ^{ - 1}}v(x,t)} \right){{(d\iota )}^{ - 1}}{|_{v(x,t)}}d\Pi {|_{v(x,t)}}\xi ,\eta } \right\rangle  \\
&= \left\langle {d\iota {|_{{\iota ^{ - 1}}(v(x,t)}}J\left( {{\iota ^{ - 1}}v(x,t)} \right){{(d\iota )}^{ - 1}}{|_{v(x,t)}}d\Pi {|_{v(x,t)}}\xi ,d\Pi {|_{v(x,t)}}\eta } \right\rangle  \\
&= \left\langle {J\left( {{\iota ^{ - 1}}v(x,t)} \right){{(d\iota )}^{ - 1}}{|_{v(x,t)}}d\Pi {|_{v(x,t)}}\xi ,{{(d\iota )}^{ - 1}}{|_{v(x,t)}}d\Pi {|_{v(x,t)}}\eta } \right\rangle  \\
&=  - \left\langle {{{(d\iota )}^{ - 1}}{|_{v(x,t)}}d\Pi {|_{v(x,t)}}\xi ,J\left( {{\iota ^{ - 1}}v(x,t)} \right){{(d\iota )}^{ - 1}}{|_{v(x,t)}}d\Pi {|_{v(x,t)}}\eta } \right\rangle  \\
&=  - \left\langle {\xi ,(d\iota ){|_{{\iota ^{ - 1}}(v(x,t)}}J\left( {{\iota ^{ - 1}}v(x,t)} \right){{(d\iota )}^{ - 1}}{|_{v(x,t)}}d\Pi {|_{v(x,t)}}\eta } \right\rangle.
\end{align*}
Thus $\Phi^*=-\Phi$, and consequently $\left( {{\gamma _1} - {\gamma _2}\Phi } \right)^ * \left( {{\gamma _1} - {\gamma _2}\Phi } \right) = \gamma _1^2 - \gamma _2^2{\Phi ^2}$. Suppose that $\lambda$ is an eigenvalue of $\gamma _1^2 - \gamma _2^2{\Phi ^2}$, $\xi$ is the corresponding eigenfunction, then the formula
$$
\gamma _1^2\xi  - \gamma _2^2{\Phi ^2}\xi  = \lambda \xi,
$$
with  $\Phi^*=-\Phi$  gives
$$\gamma _1^2{\left| \xi  \right|^2} + \gamma _2^2\left\langle {\Phi \xi ,\Phi \xi } \right\rangle  = \lambda {\left| \xi  \right|^2}.
$$
Therefore, we conclude
\begin{align}\label{FV6}
\gamma _1^2 \le \lambda  \le \gamma _2^2 + \gamma _1^2.
\end{align}
Particularly, we have
\begin{align}\label{FV7}
\left\| {{{\left( {{\gamma _1} - {\gamma _2}\Phi } \right)}^{ - 1}}} \right\| \le \frac{1}{{{\gamma _1}}}.
\end{align}
Meanwhile, let $\eta={{{\left( {{\gamma _1} - {\gamma _2}\Phi } \right)}^{ - 1}}\xi }$,
the skew-symmetry of $\Phi$ and (\ref{FV6}) yield
\begin{align}
\left\langle {{{\left( {{\gamma _1} - {\gamma _2}\Phi } \right)}^{ - 1}}\xi ,\xi } \right\rangle & = \left\langle {\eta ,\left( {{\gamma _1} - {\gamma _2}\Phi } \right)\eta } \right\rangle  = {\gamma _1}{\left\| \eta  \right\|^2} = {\gamma _1}{\left\| {{{\left( {{\gamma _1} - {\gamma _2}\Phi } \right)}^{ - 1}}\xi } \right\|^2} \nonumber\\
&\ge {\gamma _1}{\left\| \xi  \right\|^2}{\rho _{\min }}\ge \frac{{{\gamma _1}}}{{\gamma _2^2 + \gamma _1^2}}{\left\| \xi  \right\|^2} = \alpha {\left\| \xi  \right\|^2},\label{FV8}
\end{align}
where $\rho_{min}$ is the minimal eigenvalue of $(\gamma_1-\gamma_2\Phi)^*(\gamma_1-\gamma_2\Phi)$.
Lemma \ref{par} follows by (\ref{FV7}) and (\ref{FV8}).
\end{proof}

\begin{remark}
Lemma \ref{par} is of limited use in the study of dynamic behaviors, since (\ref{FV3}) is highly nonlinear and loses the nice geometric structures of (\ref{1}). However, Lemma \ref{par} reveals the parabolic nature of $(\ref{1})$ and is useful for local theorems, especially the local well-posedenss and local regularity with respect to $x$, for instance the smoothness.
\end{remark}

\begin{lemma}\label{2.2}
There exists a universal constant $c>0$, such that for any given $u\in W^{2,2}(\Bbb R^2;\mathcal{N})$, $R>0$, $x\in \Bbb R^2$, $\varphi\in L^{\infty}(B_{R}(x))$ satisfying $\varphi(y)=\varphi(|x-y|)$ for arbitrary $y\in \Bbb R^2$, we have
\begin{align}\label{2.1}
\int^T_0\int_{\Bbb R^2}|\nabla u|^4\varphi dxdt\le c\cdot\big(\mathop {{\rm{esssup}}}\limits_{_{0 \le t \le T}}\int_{B_R(x)}|\nabla u|^2dy\big)\big(\int^T_0\int_{\Bbb R^2}|\nabla^2u|^2\varphi dxdt+R^{-2}\int^T_0\int_{\Bbb R^2}|\nabla u|^2\varphi dxdt\big).
\end{align}
\end{lemma}
\begin{proof}
The proof is standard, for the completeness, we restate the proof. By the density of step functions in $L^{\infty}(B_R(x))$, it suffices to prove Lemma
\ref{2.2} for $\varphi\equiv 1$. Let $K(R,x,u)$ be the mean value of $|\nabla u|_{u^*h}$ in $B_R(x)$, then Gagliardo-Nirenberg inequality for scalar functions yields
\begin{align*}
\int^T_0\int_{B_R(x)}|\nabla u|^4dydt&\lesssim\int^T_0\int_{B_R(x)}\big||\nabla u|-K\big|^4dydt+\int^T_0\int_{B_R(x)}K^4dydt\\
&\le c\cdot\big(\mathop {ess\sup }\limits_{0 \le t \le T}\int_{B_R(x)}\big||\nabla u|-K\big|^2dy\big)\cdot\int^T_0\int_{B_R(x)}\big|\partial_x|\nabla u|\big|^2dydt\\
&\mbox{  }+\big({\rm{Vol}}(B_R(x))\big)^{-3}\int^T_0\big|\int_{B_R(x)}|\nabla u|dy\big|^4dt.
\end{align*}
It is easily seen that
\begin{align*}
\int_{B_R(x)}\big||\nabla u|-K\big|^2dy&\le \int_{B_R(x)}|\nabla u|^2dy\\
\int^T_0\big|\int_{B_R(x)}|\nabla u|dy\big|^4dt&\le \big({\rm{Vol}}(B_R(x))\big)^{2}\big(\mathop {ess\sup }\limits_{0 \le t \le T}
\int_{B_R(x)}|\nabla u|^2dy\big)\int^T_0\int_{B_R(x)}|\nabla u|^2dydt,
\end{align*}
which combined with Kato type inequality gives Lemma \ref{2.2}.
\end{proof}

A simple covering argument yields the following lemma.
\begin{lemma}\label{bb.1}
There exists a universal constant $c$ such that for any $u\in W^{2,2}(\Bbb R^2;\mathcal{N})$, $R>0$, we have
\begin{align*}
\int^T_0\int_{\Bbb R^2} |\nabla u|^4dydt\le c\cdot\big(\mathop {ess\sup }\limits_{0 \le t \le T,x \in {\Bbb R^2}}\int_{B_R(x)}|\nabla u|^2dy\big)\big(\int^T_0\int_{\Bbb R^2}|\nabla^2 u|^2dxdt+R^{-2}\int^T_0\int_{\Bbb R^2}|\nabla u|^2dxdt\big).
\end{align*}
\end{lemma}

Direct calculations and the identity $\left\langle {JX,X} \right\rangle  = 0$, for any $X\in T\mathcal{N}$ imply the following energy identity.
\begin{lemma}\label{2.3}
For any regular solution to (\ref{1}), for all $t>0$, we have
$$
 \alpha \int^t_0\int_{\Bbb R^2}|\partial_t u|^2dydt=\big(E(u_0)-E(u(t))\big),
$$
and consequently
\begin{align*}
\alpha\int^t_0\int_{\Bbb R^2}|\sum^2_{i=1}\nabla_i\partial_i u|^2dydt\lesssim \big(E(u_0)-E(u(t))\big).
\end{align*}
\end{lemma}

\begin{remark}\label{2.5}
Lemma \ref{bb.1} and Lemma \ref{2.3} give the estimate
\begin{align}\label{2.4}
\int^T_0\int_{\Bbb R^2}|\nabla u|^4dydt\lesssim \big(\|E(u(t);B_R(x))\|_{L^{\infty}_{t,x}([0,T]\times\Bbb R^2)}\big)\big(\int^T_0\int_{\Bbb R^2}|\nabla^2u|^2dydt +\frac{T}{R^2}E(u_0)\big).
\end{align}
\end{remark}

\begin{lemma}\label{2.6}
Let $\varphi$ be a smooth function in $\Bbb R^2$ which satisfies the estimate $|\nabla^k \varphi|\le C(k)\frac{1}{R^k}$ for some $R>0$.
Then there exists a universal constant $c$  depending only on $\mathcal{N}$ such that for arbitrary regular solution $u$ to (\ref{1}),
$$\int_{\Bbb R^2} {{{\left| {\nabla u(s,x)} \right|}^2}{\varphi ^2}dx}  \le \int_{\Bbb R^2} {{{\left| {\nabla u(0,x)} \right|}^2}{\varphi ^2}dx}  + \frac{{cs}}{{{R^2}}}E({u_0}).
$$
\end{lemma}
\begin{proof}
 Applying (\ref{1}), using the zero-tension property and comparable property, integration by parts and the skew-symmetry of the symplectic form, we have
\begin{align}
\frac{d}{{dt}}\sum\limits_{j = 1}^2 &{\int_{\Bbb R^2} {\left\langle {{\partial _{{x_j}}}u,{\partial _{{x_j}}}u} \right\rangle } {\varphi ^2}dx} = 2\sum\limits_{j = 1}^2 {\int_{\Bbb R^2} {\left\langle {{\nabla _t}{\partial _{{x_j}}}u,{\partial _{{x_j}}}u} \right\rangle } {\varphi ^2}dx}= 2\sum\limits_{j = 1}^2 {\int_{\Bbb R^2} {\left\langle {{\nabla _{{x_j}}}{\partial _t}u,{\partial _{{x_j}}}u} \right\rangle } {\varphi ^2}dx}  \nonumber\\
 &= 2\alpha \sum\limits_{j = 1,l = 1}^2 {\int_{\Bbb R^2} {\left\langle {{\nabla _{{x_j}}}{\nabla _{{x_l}}}{\partial _{{x_l}}}u,{\partial _{{x_j}}}u} \right\rangle } {\varphi ^2}dy}  - 2\beta \sum\limits_{j = 1,l = 1}^2 {\int_{\Bbb R^2} {\left\langle {J{\nabla _{{x_j}}}{\nabla _{{x_l}}}{\partial _{{x_l}}}u,{\partial _{{x_j}}}u} \right\rangle } {\varphi ^2}dx}  \nonumber\\
 &= 2\alpha \sum\limits_{j = 1,l = 1}^2 {\int_{\Bbb R^2} {{\varphi ^2}{\partial _{{x_j}}}} \left\langle {{\nabla _{{x_l}}}{\partial _{{x_l}}}u,{\partial _{{x_j}}}u} \right\rangle dx}  - 2\alpha \sum\limits_{j = 1,l = 1}^2 {\int_{\Bbb R^2} {{\varphi ^2}} \left\langle {{\nabla _{{x_l}}}{\partial _{{x_l}}}u,{\nabla _{{x_j}}}{\partial _{{x_j}}}u} \right\rangle dx}  \nonumber\\
 &\mbox{  }- 2\beta \sum\limits_{j = 1,l = 1}^2 {\int_{\Bbb R^2} {{\varphi ^2}{\partial _{{x_j}}}} \left\langle {J{\nabla _{{x_l}}}{\partial _{{x_l}}}u,{\partial _{{x_j}}}u} \right\rangle dx}  + 2\beta \sum\limits_{j = 1,l = 1}^2 {\int_{\Bbb R^2} {\left\langle {J{\nabla _{{x_l}}}{\partial _{{x_l}}}u,{\nabla _{{x_j}}}{\partial _{{x_j}}}u} \right\rangle {\varphi ^2}} dx}  \nonumber\\
 &=  - 2\alpha \sum\limits_{j = 1}^2 {\int_{\Bbb R^2} {\left\langle {\sum\limits_{l = 1}^2 {{\nabla _{{x_l}}}{\partial _{{x_l}}}u} ,{\partial _{{x_j}}}u} \right\rangle } {\partial _{{x_j}}}{\varphi ^2}dx}  - 2\alpha \int_{\Bbb R^2} {{\varphi ^2}} \left\langle {\sum\limits_{l = 1}^2 {{\nabla _{{x_l}}}{\partial _{{x_l}}}u} ,\sum\limits_{j = 1}^2 {{\nabla _{{x_j}}}{\partial _{{x_j}}}u} } \right\rangle dx \nonumber\\
 &\mbox{  }+ 2\beta \sum\limits_{j = 1}^2 {\int_{\Bbb R^2} {\left\langle {J\sum\limits_{l = 1}^2 {{\nabla _{{x_l}}}{\partial _{{x_l}}}u} ,{\partial _{{x_j}}}u} \right\rangle {\partial _{{x_j}}}{\varphi ^2}} dx}  + 2\beta \int_{\Bbb R^2} {{\varphi ^2}} \left\langle {J\sum\limits_{l = 1}^2 {{\nabla _{{x_l}}}{\partial _{{x_l}}}u} ,\sum\limits_{j = 1}^2 {{\nabla _{{x_j}}}{\partial _{{x_j}}}u} } \right\rangle dx \label{cvb2}\\
 &\le \frac{c}{R}\int_{\Bbb R^2} {\left| {{\partial _t}u} \right|\left| {\nabla u} \right|\varphi dx}-\alpha\int_{\Bbb R^2}|\partial_tu|^2\varphi^2dx.\label{3.0}
 \end{align}
Integrating (\ref{3.0}) with respect to $t$ in [0,s], by Young's inequality, we obtain
\begin{align}\label{3.1}
\int_{{\Bbb R^2}} {{{\left| {\nabla u(s,x)} \right|}^2}{\varphi ^2}dx}  \le \int_{{\Bbb R^2}} {{{\left| {\nabla u(0,x)} \right|}^2}{\varphi ^2}dx}  + \frac{c}{{{R^2}}}\int_0^s {\int_{{\Bbb R^2}} {{{\left| {\nabla u} \right|}^2}} dxdt}.
\end{align}
Lemma \ref{2.6} follows from the non-increasing of the energy.
\end{proof}

Using the extrinsic formulation (\ref{FV3}), we have an outer ball bound for $\|u^i-Q^i\|_2$. Without loss of generality, we can assume $Q$ is the origin of $\Bbb R^m$.
\begin{lemma}\label{fxz}
Let $\varphi$ be a smooth function in $\Bbb R^2$ which satisfies the estimate $|\nabla^k \varphi|\le C(k)\frac{1}{R^k}$ for some $R>0$.
Then there exists a universal constant $c$  depending only on $\mathcal{N}$ such that for arbitrary regular solution $u$ to (\ref{1}), in the extrinsic sense,
$$\int_{\Bbb R^2} \varphi  {\left| {u(t,x)} \right|^2}dx \le \int_{\Bbb R^2} \varphi  \left( {{{\left| {{u_0}(x)} \right|}^2} + {{\left| {\nabla {u_0}} \right|}^2}} \right)dx + \frac{{{e^{Ct}}}}{R}\left( {E\left( {{u_0}} \right) + \left\| {{u_0}} \right\|_{L_x^2}^2} \right).
$$
\end{lemma}
\begin{proof}
For a smooth non-negative function $\varphi$, by (\ref{FV3}), Lemma \ref{par} and integration by parts, we have
\begin{align}
 {\partial _t}\int_{\Bbb R^2} \varphi  {\left| u \right|^2}dx
&=2\int_{\Bbb R^2} {\varphi {u^T}} M(u)\left( {\Delta u - A(u)\left( {\nabla u,\nabla u} \right)} \right)dx \nonumber\\
& =  - 2\int_{\Bbb R^2} \varphi  {\partial _j}{u^T}M(u){\partial _j}udx + O\left(\frac{1}{R} {\int_{\Bbb R^2} {\left| {\nabla \varphi }\right|\left| {\nabla u} \right|\left| u \right|} dx} \right) + O\left( {\int_{\Bbb R^2} {\varphi \left| u \right|} {{\left| {\nabla u} \right|}^2}dx}
 \right).\label{kcx1}
\end{align}
Particularly, if $ \varphi=1$, we obtain
\begin{align*}
{\partial _t}\int_{\Bbb R^2} {{{\left| u \right|}^2}} dx \le  - 2\alpha \int_{\Bbb R^2} {{{\left| {\nabla u} \right|}^2}} dx + O\left( {\int_{\Bbb R^2} {\left| u \right|} {{\left| {\nabla u} \right|}^2}dx} \right) \lesssim \int_{\Bbb R^2} {{{\left| {Du} \right|}^2}} dx,
\end{align*}
where we use the bound $\|u\|_{L^{\infty}_x}\lesssim 1$ due to the compactness of $\mathcal{N}$.
Then $\|u\|_{L^2_x}$ has at most a linear growth with respect to $t$:
\begin{align}\label{kjhgf}
\int_{\Bbb R^2} {{{\left| {u(t,x)} \right|}^2}} dx \le \int_{\Bbb R^2} {{{\left| {u(0,x)} \right|}^2}} dx + CtE\left( {{u_0}} \right).
\end{align}
Coming back to (\ref{kcx1}), for any $\varphi$ given in Lemma \ref{fxz}, we have
\begin{align*}
 {\partial _t}\int_{\Bbb R^2} \varphi  {\left| u \right|^2}dx &\le  - 2\alpha \int_{\Bbb R^2} \varphi  {\left| {\nabla u} \right|^2}dx + O\left( \frac{1}{R} {\int_{\Bbb R^2} {\left| {\nabla u} \right|\left| u \right|} dx} \right) + O\left( {\int_{\Bbb R^2} {\varphi \left| u \right|} {{\left| {\nabla u} \right|}^2}dx} \right) \\
 &\le O\left( \frac{1}{R}{\int_{\Bbb R^2} {{{\left| {\nabla u} \right|}^2}} dx} \right) + O\left( \frac{1}{R}{\int_{\Bbb R^2} {{{\left| u \right|}^2}} dx} \right) + \int_{\Bbb R^2} \varphi  {\left| {\nabla u} \right|^2}dx,
\end{align*}
where again we use the bound $\|u\|_{L^{\infty}_x}\lesssim 1$. Thus (\ref{kjhgf}) and Lemma \ref{2.6} imply
$${\partial _t}\int_{\Bbb R^2} \varphi  {\left| u \right|^2}dx \le \frac{1}{R}E\left( {{u_0}} \right) + \frac{1}{R}\left( {\int_{\Bbb R^2} {{{\left| {u(0,x)} \right|}^2}} dx + CtE\left( {{u_0}} \right)} \right) + \left( {\int_{\Bbb R^2} \varphi  {{\left| {\nabla {u_0}} \right|}^2}dx + \frac{{ct}}{{{R^2}}}E\left( {{u_0}} \right)} \right).
$$
Integrating this formula with respect to $t$ gives Lemma \ref{fxz}.
\end{proof}

Lemma \ref{2.6} and Lemma \ref{fxz} have several useful corollaries by choosing different $\varphi$. We collect them below.
\begin{corollary}\label{saz23}
For any regular solution $u(t,x)$ to (\ref{1}), $0\le s_1<s_2<\infty$, we have
\begin{align}
&\mbox{ }E(u({s_2});{B_R}(x)) \le E(u({s_1});{B_{2R}}(x)) + \frac{{{C_3}({s_2} - {s_1})}}{{{R^2}}}E({u_0}), \label{1s2}\\
&\mbox{  }\int_{|x|\ge 2R}|u({t})|^2+|\nabla u({t})|^2 \le \int_{|x|\ge R}|u({t})|^2+|\nabla u({t})|^2+ C(t)\frac{1}{R}(E({u_0})+\|u_0\|^2_{L^2_x}), \label{1s5}\\
&\mbox{ }E(u({s_2});{B_{2R}}(x)) \ge E(u({s_1});{B_R}(x)) - C\int^{s_2}_{s_1}|\partial_t u|^2_{L^2_x}ds-\frac{{{C_3}({s_2} - {s_1})}}{{{R^2}}}E({u_0}),\label{1s3}
\end{align}
where $E(u;B_R(x))$ is the local energy defined by
$$
E(u(t);B_R(x))=\int_{|y-x|\le R}|\nabla u(t,y)|^2dy.
$$
\end{corollary}
\begin{proof}
Take $\varphi$ be a cutoff function which equals one in $B_R(x)$ and vanishes outside of $B_{2R}(x)$, then (\ref{1s2}) follows from Lemma \ref{2.6}.
Let $\varphi$ be a cutoff function which equals one outside of $B_{2R}(0)$ and vanishes inside of $B_R(0)$, then we have $(\ref{1s5})$ from Lemma \ref{fxz} and Lemma \ref{2.6}.
(\ref{1s3}) needs additional efforts. By (\ref{cvb2}), we have
\begin{align}\label{z23s}
\left| {\frac{d}{{dt}}\int_{\Bbb R^2} {\left| {\nabla u} \right|^2{\varphi ^2}dx} } \right| \le \int_{\Bbb R^2} {\left| {{\partial_t}u} \right|{\varphi ^2}dx}  + \frac{1}{{{R^2}}}\int_{\Bbb R^2} {{{\left| {\nabla u} \right|}^2}dx}.
\end{align}
Then (\ref{1s3}) follows by integrating (\ref{z23s}) respect to $t$ in $(s_1,s_2)$.
\end{proof}

Recall the definition of the function space $Y([0,T]\times\Bbb R^2)$
\begin{align*}
 &Y([0,T] \times {\mathbb{R}^2}) \\
 &\triangleq\left\{ {u:[0,T] \times {\mathbb{R}^2} \to {\mathbb{R}^m},u(t,x) \in \mathcal{N},a.e.\left| \begin{array}{l}
 u \in C([0,T];{L^2}({\mathbb{R}^2})),\nabla u \in {L^\infty }([0,T];{L^2}({\mathbb{R}^2})) \\
 {\nabla ^2}u \in {L^2}([0,T] \times {\mathbb{R}^2}),{\partial _t}u \in {L^2}([0,T] \times {\mathbb{R}^2}) \\
 \end{array} \right.} \right\}.
\end{align*}

We need a compactness lemma, namely Lemma \ref{10.2} in the proof of the local well-posedness.
\begin{lemma}\label{z345c}
If $\{f_m\}$ is bounded in $C([0,T];L^2(\Bbb R^2))\bigcap C([0,T];{\dot H}^1(\Bbb R^2))$, $\{\partial_tf_m\}$ is bounded in $L^2([0,T];L^2(\Bbb R^2))$, and for any $\varepsilon>0$, there exists $R(\varepsilon)$ such that
\begin{align}\label{jnkl}
\mathop {\sup }\limits_m\int_{|x|\ge R(\varepsilon)}|f_m|^2dx<\varepsilon,
\end{align}
then  $\{f_m\}$ is precompact in $C([0,T];L^2(\Bbb R^2))$.
\end{lemma}
\begin{proof}
By the  Arzela-Ascoli lemma, it suffices to prove $f_m(t)$ is compact in $L^2$ for any fixed $t\in[0,T]$ and $f_m(t)$ is equi-continuous in $C([0,T];L^2)$. The equi-continuity of $f_m(t)$ follows from
$$
\|f_m(t_1)-f_m(t_2)\|_{L^2}\le (t_2-t_1)^{\frac{1}{2}}\big(\int^{t_2}_{t_1}\|\partial_t f_m\|^2_{L^2_x}\big)^{\frac{1}{2}}.
$$
Since we have (\ref{jnkl}), it suffices to prove
$$
\mathop {\lim }\limits_{h \to 0} \mathop {\sup }\limits_m \int_{\Bbb R^2} {{{\left| {{f_m}(x + h) - {f_m}(x)} \right|}^2}dx}  = 0.
$$
Then by Parseval identity, it suffices to show
\begin{align}\label{zsdfc}
\mathop {\lim }\limits_{h \to 0} \mathop {\sup }\limits_m \int_{\Bbb R^2} {{{\left| {\left( {{e^{i\xi h}} - 1} \right){\widehat{f}_m}(\xi )} \right|}^2}d\xi }  = 0.
\end{align}
Indeed, by Parseval identity and the mean-value theorem, we have
\begin{align*}
&\int_{\Bbb R^2} {{{\left| {\left( {{e^{i\xi h}} - 1} \right){\widehat{f}_m}(t,\xi )} \right|}^2}d\xi }  \\
&= \int_{\left| \xi  \right| < R} {{{\left| {\left( {{e^{i\xi h}} - 1} \right){\widehat{f}_m}(\xi )} \right|}^2}d\xi }  + \int_{\left| \xi  \right| \ge R} {{{\left| {\left( {{e^{i\xi h}} - 1} \right){\widehat{f}_m}(\xi )} \right|}^2}d\xi }  \\
&\lesssim \left| h \right|R\int_{\Bbb R^2} {{{\left| {{\widehat{f}_m}(\xi )} \right|}^2}d\xi }  + \frac{1}{{{R^2}}}\int_{\Bbb R^2} {{{\left| {{\widehat{f}_m}(\xi )} \right|}^2}{{\left| \xi  \right|}^2}d\xi }  \\
&\lesssim \left( {\left| h \right|R + \frac{1}{{{R^2}}}} \right){\left\| {{{f}_m}} \right\|_{{W^{1.2}}}}.
\end{align*}
Hence for any $\varepsilon>0$, choose $R$ sufficiently large, then for an acceptable $R$, let $h$ go to $0$, (\ref{zsdfc}) follows.
\end{proof}

\begin{lemma}\label{10.2}
Assume that $\{u_m(t,x)\}$ is bounded in $Y([0,T]\times\Bbb R^2)$, for each $m\in \Bbb N^+$, $u_m(t,x)$ is a regular solution to (\ref{1}) with the initial data $u_{m,0}(x)$. If $\{u_{m,0}\}$ converges in $W^{1,2}(\Bbb R^2)$, then there exists some $u\in Y([0,T]\times\Bbb R^2)$ such that in the extrinsic sense, up to a subsequence, we have
$$
u_m\to  u, \mbox{  }{\rm{in}}\mbox{  }C([0,T];L^2(\Bbb R^2)),D u_m\to D u, \mbox{  }{\rm{in}}\mbox{  }L^2([0,T];L^2(\Bbb R^2)),
$$
and
$$
D^2 u_m\rightharpoonup D^2u, {\rm{in}}\mbox{  }L^2([0,T];L^2(\Bbb R^2)), \mbox{  }\mbox{  }\partial_tu_m\rightharpoonup \partial_tu, {\rm{in}}\mbox{  }L^2([0,T];L^2(\Bbb R^2)).
$$
\end{lemma}
\begin{proof}
The following pointwise estimate is known, for instance \cite{DW},
\begin{align}
|D u|&= |\nabla u|, \label{sw.1}\\
|D^{\rho}u|&\lesssim |\nabla^2 u|+|\nabla u|^2, \mbox{  }{\rm{for}}\mbox{  }{\rm{any}}\mbox{  }|\rho|=2. \label{sw.2}
\end{align}
Therefore, by (\ref{sw.1}), (\ref{sw.2}), Lemma \ref{bb.1}, we have
$$
\|u_m\|_{L^{\infty}([0,T];W^{1,2})}\le C, \|D^2 u_m\|_{L^2([0,T];L^2(\Bbb R^2))}\le C, \|\partial_t u_m\|_{L^2([0,T];L^2(\Bbb R^2))}\le C.
$$
Thus up to a subsequence, there exists a map $u$ from $\Bbb R^2\times [0,T]$ to $\Bbb R^m$, for which $u(t,x)\in \mathcal{N}$ for almost all $(t,x)\in [0,T]\times\Bbb R^2$, such that
\begin{align}\label{sw.5}
&u_m\rightharpoonup u, {\rm{in}}\mbox{  }L^2([0,T];{\dot H}^2(\Bbb R^2)), \mbox{  }\mbox{  }\partial_tu_m\rightharpoonup \partial_tu, {\rm{in}}\mbox{  }L^2([0,T];L^2(\Bbb R^2));\\
&u_m\mathop \rightharpoonup\limits^*  u, {\rm{in}}\mbox{  }L^{\infty}([0,T];L^2(\Bbb R^2)),\mbox{  }\mbox{  }Du_m\mathop \rightharpoonup\limits^* D u, {\rm{in}}\mbox{  }L^{\infty}([0,T];L^2(\Bbb R^2)).
\end{align}
It suffices to prove $u_m\to u$ in $L^2([0,T];\dot{H}^1(\Bbb R^2))\bigcap C([0,T];L^2(\Bbb R^2))$. Since $u_{m,0}$ converges in $W^{1,2}(\Bbb R^2)$, then for any $\varepsilon>0$ there exists $R$ sufficiently large depending only on $\varepsilon$ such that
\begin{align}\label{sw.3}
\mathop {\sup }\limits_m\int_{\left| x \right| \ge R} {{\left| {D{u_{m,0}}} \right|^2} + {\left| {{u_{m,0}}} \right|^2}
dx}  < \varepsilon.
\end{align}
Corollary \ref{saz23} yields for $t\in[0,T]$,
$$\int_{\left| x \right| \ge 2R} {{\left| {D{u_{m}(t)}} \right|^2} + {\left| {{u_{m}(t)}} \right|^2}dx}  < \int_{\left| x \right| \ge R} {{{\left| {\nabla {u_{m,0}}} \right|}^2}dx}  + \int_{\left| x \right| \ge R} {{{\left| {{u_{m,0}}} \right|}^2}dx}+\frac{C(T)}{{{R^2}}}E({u_{m,0}}),
$$
which combined with (\ref{sw.3}) gives
\begin{align} \label{sw.4}
\mathop {\sup }\limits_m \int_{\left| x \right| \ge 2R} {{\left| {D{u_{m}(t)}} \right|^2} + {\left| {{u_{m}(t)}} \right|^2}dx}  < 2\varepsilon,
\end{align}
uniformly on $t\in[0,T]$ for $R$ sufficiently large.
By Aubin-Loins lemma, we infer from (\ref{sw.5}) that
\begin{align}\label{sw.7}
\mathop {\lim }\limits_{m \to \infty } {\left\| {D{u_m} - Du} \right\|_{L^2([0,T];{L^2}({B_{3R}}))}} = 0.
\end{align}
The weak convergence of $D u_m$ to $Du$ in $L^2([0,T]\times \Bbb R^2)$ and (\ref{sw.4}) yield
\begin{align}\label{sw.8}
\|D u\|_{L^2((\Bbb R^2\backslash B_{2R})\times [0,T])}\le 2\varepsilon.
\end{align}
Thus (\ref{sw.7}) and (\ref{sw.8}) imply for $m>m_0$, we have
$$
\|Du_m-Du\|_{L^2([0,T]\times \Bbb R^2)}\le 3\varepsilon.
$$
Hence $Du_m\to Du$ in $L^2([0,T]\times \Bbb R^2)$.
It remains to prove $u_m\to u$ in $C([0,T];L^2(\Bbb R^2))$. However, this follows from Lemma \ref{z345c} and (\ref{sw.4}).
\end{proof}

For a solution $u(t,x)$ to (\ref{1}) and $R\in(0,1]$, define
$$
\varepsilon(R)=\varepsilon(R;u,T)=\mathop {\sup }\limits_{x \in {\Bbb R^2},t \in [0,T]} E(u(t);B_R(x)).
$$

\begin{lemma}\label{4.1}
There exists a universal constant $\varepsilon_1>0$ such that for any regular solution to (\ref{1}), any $R\in(0,1]$, we have
$$
\int^T_0\int_{\Bbb R^2}|\nabla^2u|^2dxdt\le cE(u_0)(1+TR^{-2}),
$$
provided $\varepsilon(R)\le \varepsilon_1$, for some $R,\varepsilon_1$.
\end{lemma}
\begin{proof}
It is obvious that there exists a decomposition $\Bbb R^2=\bigcup _{i = 1}^\infty B({x_i},R/2)$, such that for any $x\in \Bbb R^2$, there exist at most $N$ balls in the family of $\{B(x_i,R)\}$ which has a non-empty intersection with $\{x\}$. Fixed $i\in \Bbb N^+$, let $\varphi$ be a smooth function supported in $B(x_i,R)$, which equals one in $B(x_i,R/2)$ and satisfies the estimate $|\nabla \varphi|\lesssim \frac{1}{R}$. (\ref{3.0}) shows
\begin{align}\label{5.2}
\frac{d}{{dt}}\int_{\Bbb R^2} {{{\left| {\nabla u} \right|}^2}{\varphi ^2}dx}  \le  - \alpha \int_{\Bbb R^2} {\left\langle {\sum\limits_{j = 1}^2 {{\nabla _{{x_j}}}} {\partial _{{x_j}}}u,\sum\limits_{j = 1}^2 {{\nabla _{{x_j}}}} {\partial _{{x_j}}}u} \right\rangle {\varphi ^2}dx}  + \frac{c}{R}\int_{\Bbb R^2} {\left| {{\partial _t}u} \right|\left| {\nabla u} \right|\varphi dx}.
\end{align}
Integration by parts and the bounded geometric assumptions of $\mathcal{N}$ imply
\begin{align}\label{7.1}
&\int_{\Bbb R^2} {\left\langle {\sum\limits_{j = 1}^2 {{\nabla _{{x_j}}}} {\partial _{{x_j}}}u,\sum\limits_{j = 1}^2 {{\nabla _{{x_j}}}} {\partial _{{x_j}}}u} \right\rangle {\varphi ^2}dx} \nonumber\\
&= \int_{\Bbb R^2} {\sum\limits_{j = 1}^2 {{\varphi ^2}} {\partial _{{x_j}}}\left\langle {{\partial _{{x_j}}}u,\sum\limits_{l = 1}^2 {{\nabla _{{x_l}}}} {\partial _{{x_l}}}u} \right\rangle dx}  - \int_{\Bbb R^2} {\sum\limits_{j = 1}^2 {{\varphi ^2}} \left\langle {{\partial _{{x_j}}}u,\sum\limits_{l = 1}^2 {{\nabla _{{x_j}}}{\nabla _{{x_l}}}} {\partial _{{x_l}}}u} \right\rangle dx} \nonumber\\
&=  - \int_{\Bbb R^2} {\sum\limits_{j = 1}^2 {{\partial _{{x_j}}}{\varphi ^2}} \left\langle {{\partial _{{x_j}}}u,\sum\limits_{l = 1}^2 {{\nabla _{{x_l}}}} {\partial _{{x_l}}}u} \right\rangle dx}  - \int_{\Bbb R^2} {\sum\limits_{j = 1,l = 1}^2 {{\varphi ^2}} \left\langle {{\partial _{{x_j}}}u,{\nabla _{{x_l}}}{\nabla _{{x_j}}}{\partial _{{x_l}}}u} \right\rangle dx}  \nonumber\\
&\mbox{  }- \int_{\Bbb R^2} {\sum\limits_{j = 1,l = 1}^2 {{\varphi ^2}} \left\langle {{\partial _{{x_j}}}u,\mathbf{R}\left( {{\partial _{{x_j}}}u,{\partial _{{x_l}}}u} \right)({\partial _{{x_l}}}u)} \right\rangle dx} \nonumber\\
&= O\left( {\frac{1}{R}\int_{\Bbb R^2} {\left| {{\partial _t}u} \right|\left| {\nabla u} \right|\varphi dx} } \right) + O\left( {\int_{\Bbb R^2} {{{\left| {\nabla u} \right|}^4}{\varphi ^2}dx} } \right) - \int_{\Bbb R^2} {\sum\limits_{j = 1,l = 1}^2 {{\varphi ^2}} {\partial _{{x_l}}}\left\langle {{\partial _{{x_j}}}u,\sum\limits_{l = 1}^2 {{\nabla _{{x_j}}}} {\partial _{{x_l}}}u} \right\rangle dx}  \nonumber\\
&\mbox{  }+ \int_{\Bbb R^2} {\sum\limits_{j = 1,l = 1}^2 {{\varphi ^2}} \left\langle {{\nabla _{{x_l}}}{\partial _{{x_j}}}u,{\nabla _{{x_l}}}{\partial _{{x_j}}}u} \right\rangle dx}\nonumber\\
&= O\left( {\frac{1}{R}\int_{\Bbb R^2} {\left| {{\partial _t}u} \right|\left| {\nabla u} \right|\varphi dx} } \right) + O\left( {\int_{\Bbb R^2} {{{\left| {\nabla u} \right|}^4}{\varphi ^2}dx} } \right) + O\left( {\frac{1}{R}\int_{\Bbb R^2} {\left| {{\nabla ^2}u} \right|\left| {\nabla u} \right|\varphi dx} } \right)\nonumber\\
&\mbox{  }+ \int_{\Bbb R^2} {{{\left| {{\nabla ^2}u} \right|}^2}{\varphi ^2}dx}.
\end{align}
Thus, by Young's inequality,
\begin{align}
\int_{\Bbb R^2} {\left\langle {\sum\limits_{j = 1}^2 {{\nabla _{{x_j}}}} {\partial _{{x_j}}}u,\sum\limits_{j = 1}^2 {{\nabla _{{x_j}}}} {\partial _{{x_j}}}u} \right\rangle {\varphi ^2}dx}  &\ge c\int_{\Bbb R^2} {{{\left| {{\nabla ^2}u} \right|}^2}{\varphi ^2}dx}  + O\left( {\frac{1}{R}\int_{\Bbb R^2} {\left| {{\partial _t}u} \right|\left| {\nabla u} \right|\varphi dx} } \right) \nonumber\\
&\mbox{  }+ O\left( {\frac{1}{{{R^2}}}\int_{\Bbb R^2\bigcap {\rm{supp}}\varphi} {{{\left| {\nabla u} \right|}^2}dx} } \right)+O\big(\int_{\Bbb R^2}|\nabla u|^4\varphi^2dx\big).\label{5.1}
\end{align}
Integrating (\ref{5.2}) with respect to $t$ in $[0,s]$, we obtain from (\ref{5.1}) that
\begin{align}
\int_{\Bbb R^2} {{{\left| {\nabla u(s)} \right|}^2}{\varphi ^2}dx} &\le \int_{\Bbb R^2} {{{\left| {\nabla {u_0}} \right|}^2}{\varphi ^2}dx}  - c\int_0^s {\int_{\Bbb R^2} {{{\left| {{\nabla ^2}u} \right|}^2}{\varphi ^2}dxdt} }  + \int_0^s {\int_{\Bbb R^2} {{{\left| {\nabla u} \right|}^4}{\varphi ^2}dx} dt}  \nonumber\\
&\mbox{  }+ c\int_0^s {\int_{{\Bbb R^2} \cap {\rm{supp}}\varphi } {{{\left| {{\partial _t}u} \right|}^2}dx} dt}  + \frac{c}{{{R^2}}}\int_0^s {\int_{{\Bbb R^2} \cap {\rm{supp}}\varphi } {{{\left| {\nabla u} \right|}^2}dx} dt}.\label{5.3}
\end{align}
Using the definition of $\varepsilon(R)$ and Lemma 3.2, for $B_R(x_i)$, we get
\begin{align}\label{5.3}
\int_0^s {\int_{\Bbb R^2} {{{\left| {{\nabla ^2}u} \right|}^2}{\varphi ^2}dxdt} }  \le c\int_0^s {\int_{{\Bbb R^2} \cap {\rm{supp}}\varphi } {{{\left| {{\partial _t}u} \right|}^2}dx} dt}  + \frac{c}{{{R^2}}}\int_0^s {\int_{{\Bbb R^2} \cap {\rm{supp}}\varphi } {{{\left| {\nabla u} \right|}^2}dx} dt}  + \int_{\Bbb R^2} {{{\left| {\nabla {u_0}} \right|}^2}{\varphi ^2}dx}.
\end{align}
Summing up (\ref{5.3}) over the ball in the family of $\{B_{R/2}(x_i)\}$, by Lemma \ref{2.3}, we have
$$\int_0^s {\int_{\Bbb R^2} {{{\left| {{\nabla ^2}u} \right|}^2}{\varphi ^2}dxdt} }  \le cE({u_0}) + \frac{{cs}}{{{R^2}}}E({u_0}).
$$
\end{proof}

The uniform estimate of $\|\nabla u\|_{L^4_{t,x}(I\times\Bbb R^2)}$ is useful to establish the estimates of $H^k$ norms of $u$.
\begin{lemma}\label{6.1}
For any $\varepsilon>0$, $R_1\in(0,1]$, $E_0>0$, any solution sequence $\{u_m(t,x)\}$ bounded in $Y([0,T]\times\Bbb R^2)$ to (\ref{1}) with $u_m(0)$ converging in $W^{1,2}(\Bbb R^2)$, there exists a constant $\delta>0$ independent of $m$ such that if $I$ is a time interval with $|I|<\delta$, then we have
$$
\mathop {\sup }\limits_m\int_I\int_{\Bbb R^2}|\nabla u_m|^4dxdt<\varepsilon,
$$
provided $\varepsilon(R_1;u_m,T)<\varepsilon_1$, $E(u_{m0})<E_0$.
\end{lemma}
\begin{proof}

From Lemma $\ref{4.1}$, (\ref{kjhgf}), the decreasing of the energy, we have
$$
\mathop {\sup }\limits_{t \in [0,T]} \|u_m(t)\|^2_{W^{1,2}}+\int^T_0\int_{\Bbb R^2}|\partial_tu_m|^2+|\nabla^2u|^2dxdt\le c(E_0,R_1,T).
$$
Hence by Lemma \ref{10.2}, up to a subsequence, for some map $u$ from $\Bbb R^2$ to $\mathcal{N}$, we have $u_m\to u, a.e.$, $\partial_tu_m\rightharpoonup \partial_tu$, $D^2u_m\rightharpoonup D^2u$, weakly in $L^2([0,T];L^2(\Bbb R^2))$. And $D u_m\to D u$ strongly in $L^2([0,T];L^2(\Bbb R^2))$. This shows $u(t,x)\in L^2([0,T];W^{2,2}(\Bbb R^2))$ is a weak solution to (\ref{1}). Moreover, $D u\in C([0,T];L^2(\Bbb R^2))$ by $\partial_t u, D^2u\in L^2([0,T];L^2(\Bbb R^2))$. Hence because of the compactness of $[0,T]$, $D u\in C([0,T];L^2(\Bbb R^2))$ implies for each $\varepsilon >0$, there exits a $R>0$ such that
$$
\varepsilon(2R;u,T)<\varepsilon.
$$
Since $D u_m\to Du$ in $L^2(\Bbb R^2)$ for a.e. $t\in [0,T]$, we can choose $0\le t_1<...<t_L\le T$ such that for $l\in\{1,...,L\}$
$$
|t_{l+1}-t_l|\le \frac{\varepsilon R^2}{cE_0},
$$
and some $m_0>0$ such that when $m>m_0$,
$$\int^T_0\int_{B_R(y)}|\nabla u_m(\cdot,t_l)-\nabla u(\cdot,t_l)|^2dx<2\varepsilon,
$$
uniformly for $y\in \Bbb R^2$.
By Corollary \ref{saz23}, for any $m>m_0$, $t\in(t_l,t_{l+1})$, we have
\begin{align*}
E_{R/2}(u_m(\cdot,t);x_i)&\le E_{R}(u_m(\cdot,t_l),x_i)+c\frac{t-t_l}{R^2}E_0\\
&\le E_R(u(\cdot,t_l),x_i)+3\varepsilon\le 4\varepsilon.
\end{align*}
Take the covering of $\Bbb R^2$ as in Lemma \ref{4.1}, for any $B_{R}(x_i)$ in this decomposition, let $\varphi$ be a smooth function which is supported in
$B_{R}(x_i)$ and equals one in $B_{R/2}(x_i)$.
Then Lemma \ref{2.1} implies
\begin{align}\label{6.2}
\int_I\int_{\Bbb R^2}|\nabla u_m|^4\varphi^2dxdt\le c\varepsilon\big(\int_I\int_{\Bbb R^2}|\nabla^2 u_m|^2\varphi^2dxdt+\frac{1}{R^2}\int_I\int_{\Bbb R^2} |\nabla u|^2\varphi^2dxdt\big).
\end{align}
Summing (\ref{6.2}) over all $B_R(x_i)$, for $|I|<R^2$, we conclude
$$
\int_I\int_{\Bbb R^2}|\nabla u|^4dxdt\le c\varepsilon\big(\int_I\int_{\Bbb R^2}|\nabla^2u|^2dxdt+\frac{1}{R^2}|I|E_0\big)
\le c_1\varepsilon,
$$
where $c_1$ depends only on $\mathcal{N},E_0,R_1,T$, thus Lemma \ref{6.1} follows.
\end{proof}

\begin{lemma}\label{h.4}
Let $\{u_m\}$ be a sequence of regular solutions to (\ref{1}) which is bounded in $Y([0,T]\times\Bbb R^2)$ with $u_m(0,x)$ converging in $W^{1,2}(\Bbb R^2)$, $E(u_m(0,x))\le E_0$. Then for any $0<\tau<T$, there exits some constant $C(E_0,T,\tau,R)$ such that for $t\in [\tau,T)$
$$
\mathop {\sup }\limits_m \|\nabla^2u_m(t)\|_{L^2(\Bbb R^2)}\le C(E_0,T,\tau,R).
$$
provided $\varepsilon(u_m,R)\le \varepsilon_1$ for some $R>0$.
\end{lemma}
\begin{proof}
In the following proof, we use $u$ instead of $u_m$, but all the constants are independent of $m$.
Applying (\ref{1}), comparable property, integration by parts, bounded geometric assumptions of $\mathcal{N}$, skew-symmetry of sympletic form, we obtain
\begin{align*}
&\frac{1}{2}\frac{d}{{dt}}\int_{\Bbb R^2} {\left\langle {{\partial _t}u,{\partial _t}u} \right\rangle dx}\\
&= \alpha \sum\limits_{j = 1}^2 {\int_{\Bbb R^2} {\left\langle {{\nabla _t}{\nabla _{{x_j}}}{\partial _{{x_j}}}u,{\partial _t}u} \right\rangle } dx}  - \beta \sum\limits_{j = 1}^2 {\int_{\Bbb R^2} {\left\langle {J{\nabla _t}{\nabla _{{x_j}}}{\partial _{{x_j}}}u,{\partial _t}u} \right\rangle dx} }  \\
&= \alpha \sum\limits_{j = 1}^2 {\int_{\Bbb R^2} {\left\langle {{\nabla _{{x_j}}}{\nabla _t}{\partial _{{x_j}}}u,{\partial _t}u} \right\rangle } dx}  + \alpha \sum\limits_{j = 1}^2 {\int_{\Bbb R^2} {\left\langle {\mathbf{R}\left( {{\partial _t}u,{\partial _{{x_j}}}u} \right)({\partial _{{x_j}}}u),{\partial _t}u} \right\rangle } dx}  \\
&\mbox{  }- \beta \sum\limits_{j = 1}^2 {\int_{\Bbb R^2} {\left\langle {J{\nabla _{{x_j}}}{\nabla _t}{\partial _{{x_j}}}u,{\partial _t}u} \right\rangle } dx}  - \beta \sum\limits_{j = 1}^2 {\int_{\Bbb R^2} {\left\langle {J\mathbf{R}\left( {{\partial _t}u,{\partial _{{x_j}}}u} \right)({\partial _{{x_j}}}u),{\partial _t}u} \right\rangle } dx}  \\
&=  - \alpha \sum\limits_{j = 1}^2 {\int_{\Bbb R^2} {\left\langle {{\nabla _t}{\partial _{{x_j}}}u,{\nabla _{{x_j}}}{\partial _t}u} \right\rangle } dx}  + O\left( {\int_{\Bbb R^2} {{{\left| {\nabla u} \right|}^2}{{\left| {{\partial _t}u} \right|}^2}} dx} \right) + \beta \sum\limits_{j = 1}^2 {\int_{\Bbb R^2} {\left\langle {J{\nabla _t}{\partial _{{x_j}}}u,{\nabla _{{x_j}}}{\partial _t}u} \right\rangle } dx}  \\
&=  - \alpha \sum\limits_{j = 1}^2 {\int_{\Bbb R^2} {\left\langle {{\nabla _t}{\partial _{{x_j}}}u,{\nabla _{{x_j}}}{\partial _t}u} \right\rangle } dx}  + O\left( {\int_{\Bbb R^2} {{{\left| {\nabla u} \right|}^2}{{\left| {{\partial _t}u} \right|}^2}} dx} \right).
\end{align*}
Integrating the above inequality with respect to $t$ in $(s,r)\subset[\tau,T]$, we get
\begin{align}\label{8.0}
\int_{\Bbb R^2} {{{\left| {{\partial _t}u} \right|}^2}} (r,x)dx + \alpha \int_s^r {\int_{\Bbb R^2} {{{\left| {\nabla {\partial _t}u} \right|}^2}} } dxdt \le \int_{\Bbb R^2} {{{\left| {{\partial _t}u} \right|}^2}(s,x)} dx + O\left( {\int_s^r {\int_{\Bbb R^2} {{{\left| {\nabla u} \right|}^2}{{\left| {{\partial _t}u} \right|}^2}} } dxdt} \right).
\end{align}
For $|s-r|\le1$, similar arguments as Lemma \ref{2.2} yield
\begin{align}
 \int_s^r {\int_{\Bbb R^2} {{{\left| {\nabla u} \right|}^2}{{\left| {{\partial _t}u} \right|}^2}} } dxdt &\le {\left( {\int_s^r {\int_{\Bbb R^2} {{{\left| {\nabla u} \right|}^4}} } dxdt} \right)^{1/2}}{\left( {\int_s^r {\int_{\Bbb R^2} {{{\left| {{\partial _t}u} \right|}^4}} } dxdt} \right)^{1/2}} \nonumber\\
 &\lesssim {\left( {\int_s^r {\int_{\Bbb R^2} {{{\left| {\nabla u} \right|}^4}} } dxdt} \right)^{1/2}}\left( {\mathop {ess\sup }\limits_{s \le \theta  \le r} \int_{\Bbb R^2} {{{\left| {{\partial _t}u}(\theta,x) \right|}^2}dx}  + \int_s^r {\int_{\Bbb R^2} {{{\left| {\nabla {\partial _t}u} \right|}^2}} } dxdt} \right).\label{p.1}
\end{align}
Lemma \ref{6.1} implies that for $\eta\ll1$, there exists $\delta>0$ independent of $m$, such that for $|s-r|<\delta$, we have $\|u\|_{L^4([s,r];L^4(\Bbb R^2))}<\eta$. Thus (\ref{p.1}) and (\ref{8.0}) give
$$\int_{\Bbb R^2} {{{\left| {{\partial _t}u} \right|}^2}} (t,x)dx \lesssim \mathop {\inf }\limits_{t - \delta  \le s \le t,s \ge 0} \int_{\Bbb R^2} {{{\left| {{\partial _t}u} \right|}^2}} (s,x)dx.
$$
Thus estimating the infimum by the mean value yields
$$\int_{\Bbb R^2} {{{\left| {{\partial _t}u} \right|}^2}} (t,x)dx \le C(\tau,T)E({u_0}),$$
for all $t\in(2\tau, T)$.
Then by (\ref{1}), we deduce
\begin{align}\label{ab.3}
\int_{\Bbb R^2} {{{\left| {\sum\limits_{i = 1}^2 {{\nabla _{{x_i}}}{\partial _{{x_i}}}} u} \right|}^2}} (t,x)dx \le C(\tau,T)E({u_0}).
\end{align}
For a fixed $s\in [\tau, T]$, let $v(x,t)=u(s,x)$ for $(t,x)\in[0,T]\times\Bbb R^2$, then applying Lemma 3.3 to $v$ gives
\begin{align}
\int_{\Bbb R^2} {{{\left| {\nabla u} \right|}^4}} dx &\le c\left( {\mathop {ess\sup }\limits_{x \in {R^2}} \int_{{B_R}(x)} {{{\left| {\nabla u(t,y)} \right|}^2}dy} } \right)\left( {\int_{\Bbb R^2} {{{\left| {{\nabla ^2}u(t,x)} \right|}^2}dx}  + {R^{ - 2}}\int_{\Bbb R^2} {{{\left| {\nabla u(t,x)} \right|}^2}dx} } \right) \nonumber\\
&\le c\left( {\mathop {ess\sup }\limits_{(t,x) \in (\tau ,T] \times {R^2}} \int_{{B_R}(x)} {{{\left| {\nabla u(t,y)} \right|}^2}dy} } \right)\left( {\int_{\Bbb R^2} {{{\left| {{\nabla ^2}u(t,x)} \right|}^2}dx}  + {R^{ - 2}}E\left( {{u_0}} \right)} \right). \label{ab.1}
\end{align}
Integrating by parts yields
\begin{align}
\sum\limits_{i,j = 1}^2 {\int_{\Bbb R^2} {\left\langle {{\nabla _{{x_i}}}{\partial _{{x_j}}}u,{\nabla _{{x_i}}}{\partial _{{x_j}}}u} \right\rangle } dx}  \lesssim \sum\limits_{i,j = 1}^2 {\int_{\Bbb R^2} {\left\langle {{\nabla _{{x_i}}}{\partial _{{x_i}}}u,{\nabla _{{x_j}}}{\partial _{{x_j}}}u} \right\rangle } dx}  + \int_{\Bbb R^2} {{{\left| {\nabla u} \right|}^4}} dx.\label{ab.2}
\end{align}
Since $\varepsilon(u_m,R)<\varepsilon_1$, for $\varepsilon_1$ sufficiently small, by (\ref{ab.1}), (\ref{ab.2}), (\ref{ab.3}), we obtain
$$\int_{\Bbb R^2} {\left| {{\nabla ^2}u(t)} \right|} dx \le C(\tau ,T,{E_0}).
$$
Thus Lemma \ref{h.4} follows.
\end{proof}

\begin{lemma}\label{j.2}
Let $\{u_m(t,x)\}$ which is bounded in $Y([0,T]\times\Bbb R^2)$ be solutions to $(\ref{1})$ with $u_m(0,x)$ converging in $W^{1,2}(\Bbb R^2)$, then for $\tau>0$ and any $t\in(\tau,T)$, we have
$$
\mathop {\sup }\limits_m\|D^3 u_m(t)\|_{L^2(\Bbb R^2)}\le C(\tau,T,E_0).
$$
 $\varepsilon(u_m,R)\le \varepsilon_1$ for some $R>0$.
\end{lemma}
\begin{proof}
By \cite{DW}, $\|\nabla u\|_{\mathcal{W}^{2,2}}$ is equivalent to $\|D u\|_{W^{2,2}}$, thus it suffices to bound $\|\nabla^3u_m(t)\|_{L^2(\Bbb R^2)}$.
We use $u$ instead of $u_m$ as before, and the constants are independent of $m$.
Integration by parts gives
\begin{align*}
 &\sum\limits_{i,j,l = 1}^2 {\int_{\Bbb R^2} {\left\langle {{\nabla _{{x_i}}}{\nabla _{{x_j}}}{\partial _{{x_j}}}u,{\nabla _{{x_i}}}{\nabla _{{x_l}}}{\partial _{{x_l}}}u} \right\rangle } dx}  \\
 &= \sum\limits_{i,j,l = 1}^2 {\int_{\Bbb R^2} {\left\langle {{\nabla _{{x_j}}}{\nabla _{{x_i}}}{\partial _{{x_j}}}u,{\nabla _{{x_l}}}{\nabla _{{x_i}}}{\partial _{{x_l}}}u} \right\rangle } dx}  + O\left( {\int_{\Bbb R^2} {{{\left| {\nabla u} \right|}^6}} dx} \right) + O\left( {\int_{\Bbb R^2} {{{\left| {\nabla u} \right|}^3}\left| {{\nabla ^3}u} \right|} dx} \right) \\
 &=  - \sum\limits_{i,j,l = 1}^2 {\int_{\Bbb R^2} {\left\langle {{\nabla _{{x_i}}}{\partial _{{x_j}}}u,{\nabla _{{x_j}}}{\nabla _{{x_l}}}{\nabla _{{x_i}}}{\partial _{{x_l}}}u} \right\rangle } dx}  + O\left( {\int_{\Bbb R^2} {{{\left| {\nabla u} \right|}^6}} dx} \right) + O\left( {\int_{\Bbb R^2} {{{\left| {\nabla u} \right|}^3}\left| {{\nabla ^3}u} \right|} dx} \right) \\
 &=  - \sum\limits_{i,j,l = 1}^2 {\int_{\Bbb R^2} {\left\langle {{\nabla _{{x_i}}}{\partial _{{x_j}}}u,{\nabla _{{x_l}}}{\nabla _{{x_j}}}{\nabla _{{x_i}}}{\partial _{{x_l}}}u} \right\rangle } dx}  + O\left( {\int_{\Bbb R^2} {{{\left| {\nabla u} \right|}^6}} dx} \right) + O\left( {\int_{\Bbb R^2} {{{\left| {\nabla u} \right|}^2}} {{\left| {{\nabla ^2}u} \right|}^2}dx} \right) \\
 &\mbox{  }+ O\left( {\int_{\Bbb R^2} {{{\left| {\nabla u} \right|}^3}\left| {{\nabla ^3}u} \right|} dx} \right) \\
 &= \sum\limits_{i,j,l = 1}^2 {\int_{\Bbb R^2} {\left\langle {{\nabla _{{x_l}}}{\nabla _{{x_i}}}{\partial _{{x_j}}}u,{\nabla _{{x_j}}}{\nabla _{{x_i}}}{\partial _{{x_l}}}u} \right\rangle } dx}  + O\left( {\int_{\Bbb R^2} {{{\left| {\nabla u} \right|}^6}} dx} \right) + O\left( {\int_{\Bbb R^2} {{{\left| {\nabla u} \right|}^2}} {{\left| {{\nabla ^2}u} \right|}^2}dx} \right) \\
 &\mbox{  }+ O\left( {\int_{\Bbb R^2} {{{\left| {\nabla u} \right|}^3}\left| {{\nabla ^3}u} \right|} dx} \right) \\
 &= \sum\limits_{i,j,l = 1}^2 {\int_{\Bbb R^2} {\left\langle {{\nabla _{{x_l}}}{\nabla _{{x_j}}}{\partial _{{x_i}}}u,{\nabla _{{x_l}}}{\nabla _{{x_j}}}{\partial _{{x_i}}}u} \right\rangle } dx}  + O\left( {\int_{\Bbb R^2} {{{\left| {\nabla u} \right|}^6}} dx} \right) + O\left( {\int_{\Bbb R^2} {{{\left| {\nabla u} \right|}^2}} {{\left| {{\nabla ^2}u} \right|}^2}dx} \right) \\
 &\mbox{  }+ O\left( {\int_{\Bbb R^2} {{{\left| {\nabla u} \right|}^3}\left| {{\nabla ^3}u} \right|} dx} \right).
 \end{align*}
Then we have from Young's inequality that
\begin{align}\label{g.1}
\alpha^2\int_{\Bbb R^2} {{{\left| {{\nabla ^3}u} \right|}^2}} dx \le \int_{\Bbb R^2} {{{\left| {\nabla {\partial _t}u} \right|}^2}} dx + O\left( {\int_{\Bbb R^2} {{{\left| {\nabla u} \right|}^6}} dx} \right) + O\left( {\int_{\Bbb R^2} {{{\left| {\nabla u} \right|}^2}} {{\left| {{\nabla ^2}u} \right|}^2}dx} \right).
\end{align}
Similar arguments imply
\begin{align*}
&\sum\limits_{j,l = 1}^2 2 \int_{\Bbb R^2} {\left\langle {{\nabla _{{x_l}}}{\nabla _{{x_l}}}{\partial _t}u,{\nabla _{{x_j}}}{\nabla _{{x_j}}}{\partial _t}u} \right\rangle } dx \\
&=  - \sum\limits_{j,l = 1}^2 2 \int_{\Bbb R^2} {\left\langle {{\nabla _{{x_l}}}{\partial _t}u,{\nabla _{{x_l}}}{\nabla _{{x_j}}}{\nabla _{{x_j}}}{\partial _t}u} \right\rangle } dx \\
&=  - \sum\limits_{j,l = 1}^2 2 \int_{\Bbb R^2} {\left\langle {{\nabla _{{x_l}}}{\partial _t}u,{\nabla _{{x_j}}}{\nabla _{{x_l}}}{\nabla _{{x_j}}}{\partial _t}u} \right\rangle } dx - \sum\limits_{j,l = 1}^2 2 \int_{\Bbb R^2} {\left\langle {{\nabla _{{x_l}}}{\partial _t}u,\mathbf{R}\left( {{\partial _{{x_l}}}u,{\partial _{{x_j}}}u} \right){\nabla _{{x_j}}}{\partial _t}u} \right\rangle } dx \\
&= \sum\limits_{j,l = 1}^2 2 \int_{\Bbb R^2} {\left\langle {{\nabla _{{x_j}}}{\nabla _{{x_l}}}{\partial _t}u,{\nabla _{{x_j}}}{\nabla _{{x_l}}}{\partial _t}u} \right\rangle } dx + O\left( {\int_{\Bbb R^2} {\left| {{\nabla ^2}{\partial _t}u} \right|{{\left| {\nabla u} \right|}^2}\left| {{\partial _t}u} \right|} dx} \right) + O\left( {\int_{\Bbb R^2} {{{\left| {\nabla {\partial _t}u} \right|}^2}{{\left| {\nabla u} \right|}^2}} dx} \right).
\end{align*}
Therefore, we conclude
\begin{align}\label{g.2}
\sum\limits_{j,l = 1}^2 2 \int_{\Bbb R^2} {\left\langle {{\nabla _{{x_l}}}{\nabla _{{x_l}}}{\partial _t}u,{\nabla _{{x_j}}}{\nabla _{{x_j}}}{\partial _t}u} \right\rangle } dx \ge \int_{\Bbb R^2} {{{\left| {{\nabla ^2}{\partial _t}u} \right|}^2}} dx + O\left( {\int_{\Bbb R^2} {{{\left| {\nabla u} \right|}^4}{{\left| {{\partial _t}u} \right|}^2}} dx} \right) + O\left( {\int_{\Bbb R^2} {{{\left| {\nabla {\partial _t}u} \right|}^2}{{\left| {\nabla u} \right|}^2}} dx} \right).
\end{align}
By careful calculations, we deduce
\begin{align*}
&\sum\limits_{i = 1}^2 {\frac{d}{{dt}}} \int_{\Bbb R^2} {\left\langle {{\nabla _{{x_i}}}{\partial _t}u,{\nabla _{{x_i}}}{\partial _t}u} \right\rangle } dx \\
&= \sum\limits_{i = 1}^2 2 \int_{\Bbb R^2} {\left\langle {{\nabla _t}{\nabla _{{x_i}}}{\partial _t}u,{\nabla _{{x_i}}}{\partial _t}u} \right\rangle } dx \\
&= \sum\limits_{i = 1}^2 2 \int_{\Bbb R^2} {\left\langle {{\nabla _{{x_i}}}{\nabla _t}{\partial _t}u,{\nabla _{{x_i}}}{\partial _t}u} \right\rangle } dx + \sum\limits_{i = 1}^2 2 \int_{\Bbb R^2} {\left\langle {\mathbf{R}\left( {{\partial _{{x_i}}}u,{\partial _t}u} \right){\partial _t}u,{\nabla _{{x_i}}}{\partial _t}u} \right\rangle } dx \\
&= \alpha \sum\limits_{i = 1}^2 2 \int_{\Bbb R^2} {\left\langle {{\nabla _{{x_i}}}{\nabla _t}{\nabla _{{x_j}}}{\partial _{{x_j}}}u,{\nabla _{{x_i}}}{\partial _t}u} \right\rangle } dx - \beta \sum\limits_{i = 1}^2 2 \int_{\Bbb R^2} {\left\langle {J{\nabla _{{x_i}}}{\nabla _t}{\nabla _{{x_j}}}{\partial _{{x_j}}}u,{\nabla _{{x_i}}}{\partial _t}u} \right\rangle } dx  \\
&\mbox{  }+\sum\limits_{i = 1}^2 2 \int_{\Bbb R^2} {\left\langle {\mathbf{R}\left( {{\partial _{{x_i}}}u,{\partial _t}u} \right){\partial _t}u,{\nabla _{{x_i}}}{\partial _t}u} \right\rangle } dx \\
&= \alpha \sum\limits_{i = 1}^2 2 \int_{\Bbb R^2} {\left\langle {{\nabla _{{x_i}}}{\nabla _{{x_j}}}{\nabla _{{x_j}}}{\partial _t}u,{\nabla _{{x_i}}}{\partial _t}u} \right\rangle } dx + \beta \sum\limits_{i = 1}^2 2 \int_{\Bbb R^2} {\left\langle {J{\nabla _t}{\nabla _{{x_j}}}{\partial _{{x_j}}}u,{\nabla _{{x_i}}}{\nabla _{{x_i}}}{\partial _t}u} \right\rangle } dx +  \\
&\mbox{  }+ \sum\limits_{i = 1}^2 2 \int_{\Bbb R^2} {\left\langle {\mathbf{R}\left( {{\partial _{{x_i}}}u,{\partial _t}u} \right){\partial _t}u,{\nabla _{{x_i}}}{\partial _t}u} \right\rangle } dx + \alpha \sum\limits_{i = 1}^2 2 \int_{\Bbb R^2} {\left\langle {{\nabla _{{x_i}}}\left[ {\mathbf{R}\left( {{\partial _t}u,{\partial _{{x_j}}}u} \right){\partial _{{x_j}}}u} \right],{\nabla _{{x_i}}}{\partial _t}u} \right\rangle } dx \\
&=  - \alpha \sum\limits_{i,j = 1}^2 2 \int_{\Bbb R^2} {\left\langle {{\nabla _{{x_j}}}{\nabla _{{x_j}}}{\partial _t}u,{\nabla _{{x_i}}}{\nabla _{{x_i}}}{\partial _t}u} \right\rangle } dx + \beta \sum\limits_{i,j = 1}^2 2 \int_{\Bbb R^2} {\left\langle {J{\nabla _{{x_j}}}{\nabla _t}{\partial _{{x_j}}}u,{\nabla _{{x_i}}}{\nabla _{{x_i}}}{\partial _t}u} \right\rangle } dx \\
&\mbox{  }+ \sum\limits_{i = 1}^2 2 \int_{\Bbb R^2} {\left\langle {\mathbf{R}\left( {{\partial _{{x_i}}}u,{\partial _t}u} \right){\partial _t}u,{\nabla _{{x_i}}}{\partial _t}u} \right\rangle } dx - \alpha \sum\limits_{i,j = 1}^2 2 \int_{\Bbb R^2} {\left\langle { {\mathbf{R}\left( {{\partial _t}u,{\partial _{{x_j}}}u} \right){\partial _{{x_j}}}u} ,{\nabla _{{x_i}}}{\nabla _{{x_i}}}{\partial _t}u} \right\rangle } dx \\
&\mbox{  }+ \beta \sum\limits_{i,j = 1}^2 2 \int_{\Bbb R^2} {\left\langle {J\mathbf{R}\left( {{\partial _t}u,{\partial _{{x_j}}}u} \right){\partial _{{x_j}}}u,{\nabla _{{x_i}}}{\nabla _{{x_i}}}{\partial _t}u} \right\rangle } dx \\
&=  - \alpha \int_{\Bbb R^2} {\left\langle {\sum\limits_{j = 1}^2 {{\nabla _{{x_j}}}{\nabla _{{x_j}}}{\partial _t}u} ,\sum\limits_{j = 1}^2 {{\nabla _{{x_j}}}{\nabla _{{x_j}}}{\partial _t}u} } \right\rangle } dx + O\left( {\int_{\Bbb R^2} {{{\left| {{\partial _t}u} \right|}^2}\left| {\nabla u} \right|\left| {\nabla {\partial _t}u} \right|} dx} \right) + O\left( {\int_{\Bbb R^2} {{{\left| {{\partial _t}u} \right|}^2}{{\left| {\nabla u} \right|}^4}} dx} \right).
\end{align*}
Integrating the above inequality with respect to $t$ in $[\tau,s]$ gives
\begin{align}
 \int_{\Bbb R^2} {{{\left| {\nabla {\partial _t}u} \right|}^2}(s)} dx + \frac{\alpha }{2}\int_\tau ^s {\int_{\Bbb R^2} {{{\left| {{\nabla ^2}{\partial _t}u} \right|}^2}} dxdt}  &\le \int_{\Bbb R^2} {{{\left| {\nabla {\partial _t}u} \right|}^2}(\tau )} dx + c\int_\tau ^s {\int_{\Bbb R^2} {{{\left| {{\partial _t}u} \right|}^2}\left| {\nabla u} \right|\left| {\nabla {\partial _t}u} \right|} dxdt}   \nonumber\\
 &\mbox{  }+ c\int_\tau ^s {\int_{\Bbb R^2} {{{\left| {{\partial _t}u} \right|}^2}{{\left| {\nabla u} \right|}^4}} dxdt}  + c\int_\tau ^s {\int_{\Bbb R^2} {{{\left| {\nabla {\partial _t}u} \right|}^2}{{\left| {\nabla u} \right|}^2}} dxdt}, \label{h.2}
\end{align}
For $|\tau-s|<\delta$, by Lemma \ref{6.1}, Gagliardo-Nirenberg inequality, Lemma \ref{h.4}, we obtain
\begin{align}
 &\int_\tau ^s {\int_{\Bbb R^2} {{{\left| {{\partial _t}u} \right|}^2}} } \left| {\nabla u} \right|\left| {\nabla {\partial _t}u} \right|dxdt \nonumber\\
 &\le {\left( {\mathop {\sup }\limits_{\tau  \le \theta  \le s} \int_{\Bbb R^2} {{{\left| {\nabla {\partial _t}u} \right|}^2}(\theta ,x)dx} } \right)^{1/2}}{\left( {\int_\tau ^s {\int_{\Bbb R^2} {{{\left| {\nabla u} \right|}^4}dxdt} } } \right)^{1/4}}{\left( {\int_\tau ^s {\int_{\Bbb R^2} {{{\left| {{\partial _t}u} \right|}^8}dxdt} } } \right)^{1/4}} \nonumber\\
 &\le \eta {\left( {\mathop {\sup }\limits_{\tau  \le \theta  \le s} \int_{\Bbb R^2} {{{\left| {\nabla {\partial _t}u} \right|}^2}(\theta ,x)dx} } \right)^{1/2}}{\left( {\int_\tau ^s {\int_{\Bbb R^2} {{{\left| {{\partial _t}u} \right|}^2}dxdt} } } \right)^{1/4}}{\left( {\int_\tau ^s {\int_{\Bbb R^2} {{{\left| {\nabla {\partial _t}u} \right|}^2}dxdt} } } \right)^{3/4}} \nonumber\\
 &\le C(\mu ,T,{E_0}) + {\eta ^2}\left( {\mathop {\sup }\limits_{\tau  \le \theta  \le s} \int_{\Bbb R^2} {{{\left| {\nabla {\partial _t}u} \right|}^2}(\theta ,x)dx} } \right).\label{h.5}
\end{align}
Similarly, we have
\begin{align}
&\int_\tau ^s {\int_{\Bbb R^2} {{{\left| {{\partial _t}u} \right|}^2}} } {\left| {\nabla u} \right|^4}dxdt \nonumber\\
&\le {\left( {\int_\tau ^s {\int_{\Bbb R^2} {{{\left| {\nabla u} \right|}^8}} } dxdt} \right)^{1/2}}{\left( {\int_\tau ^s {\int_{\Bbb R^2} {{{\left| {{\partial _t}u} \right|}^4}} } dxdt} \right)^{1/2}} \nonumber\\
&\le {\left( {\int_\tau ^s {{{\left( {\int_{\Bbb R^2} {{{\left| {{\nabla ^2}u} \right|}^2}dx} } \right)}^2}} dt} \right)^{1/2}}{\left( {\int_\tau ^s {\int_{\Bbb R^2} {{{\left| {\nabla u} \right|}^4}dx} } dt} \right)^{1/2}}\left( {\mathop {\sup }\limits_{\tau  \le \theta  \le s} \int_{\Bbb R^2} {{{\left| {{\partial _t}u} \right|}^2}(\theta ,x)dx + \int_\tau ^s {\int_{\Bbb R^2} {{{\left| {\nabla {\partial _t}u} \right|}^2}} } dxdt} } \right) \nonumber\\
&\le C(\mu ,T,{E_0}), \label{h.6}
\end{align}
and
\begin{align}
&\int_\tau ^s {\int_{\Bbb R^2} {{{\left| {\nabla {\partial _t}u} \right|}^2}} } {\left| {\nabla u} \right|^2}dxdt \nonumber\\
&\le {\left( {\int_\tau ^s {\int_{\Bbb R^2} {{{\left| {\nabla u} \right|}^4}} } dxdt} \right)^{1/2}}{\left( {\int_\tau ^s {\int_{\Bbb R^2} {{{\left| {\nabla {\partial _t}u} \right|}^4}} } dxdt} \right)^{1/2}} \nonumber\\
&\le \eta \left( {\int_\tau ^s {\int_{\Bbb R^2} {{{\left| {{\nabla ^2}{\partial _t}u} \right|}^2}} } dxdt + \mathop {\sup }\limits_{\tau  \le \theta  \le s} \int_{\Bbb R^2} {{{\left| {\nabla {\partial _t}u} \right|}^2}(\theta ,x)dx} } \right) \nonumber\\
&\le \eta \int_\tau ^s {\int_{\Bbb R^2} {{{\left| {{\nabla ^2}{\partial _t}u} \right|}^2}} } dxdt + \eta \mathop {\sup }\limits_{\tau  \le \theta  \le s} \int_{\Bbb R^2} {{{\left| {\nabla {\partial _t}u} \right|}^2}(\theta ,x)dx}  + C(\mu ,T,{E_0}).\label{h.7}
\end{align}
Therefore, (\ref{g.2}), (\ref{h.2}), (\ref{h.5}), (\ref{h.6}), (\ref{h.7}) imply
$$\mathop {\sup }\limits_{\tau  \le \theta  \le s} \int_{\Bbb R^2} {{{\left| {\nabla {\partial _t}u} \right|}^2}(\theta ,x)dx}  \le \mathop {\inf }\limits_{s - \delta  \le \tau  \le \theta  \le s} \int_{\Bbb R^2} {{{\left| {\nabla {\partial _t}u} \right|}^2}(\theta ,x)dx}  + C(\mu ,T,{E_0}).
$$
From Lemma \ref{h.4}, we deduce
$$\mathop {\inf }\limits_{s - \delta  \le \tau  \le \theta  \le s} \int_{\Bbb R^2} {{{\left| {\nabla {\partial _t}u} \right|}^2}(\theta ,x)dx}  \le \frac{1}{\delta }\int_\mu ^T {\int_{\Bbb R^2} {{{\left| {\nabla {\partial _t}u} \right|}^2}(\theta ,x)dx} }
\le \frac{1}{\delta }C(\mu ,T,{E_0}).
$$
Hence, we conclude
\begin{align*}
\mathop {\sup }\limits_{\tau  \le \theta  \le s} \int_{\Bbb R^2} {{{\left| {\nabla {\partial _t}u} \right|}^2}(\theta ,x)dx}  \le C(\mu ,T,{E_0}),
\end{align*}
which combined with (\ref{g.1}) yields
\begin{align*}
\int_{\Bbb R^2} {{{\left| {{\nabla ^3}u} \right|}^2}(t,x)dx}  \le C(\mu ,T,{E_0}) + c\int_{\Bbb R^2} {{{\left| {\nabla u} \right|}^6}(t,x)dx}  + c\int_{\Bbb R^2} {{{\left| {\nabla u} \right|}^2}{{\left| {{\nabla ^2}u} \right|}^2}(t,x)dx}.
\end{align*}
Using Lemma \ref{6.1} and similar arguments as before, we have
$$\int_{\Bbb R^2} {{{\left| {{\nabla ^3}u} \right|}^2}(t,x)dx}  \le C(\mu ,T,{E_0}).$$
\end{proof}

\begin{corollary}\label{j.4}
Let $\{u_m\}$ be regular solutions to (\ref{1}) bounded in $Y([0,T]\times\Bbb R^2)$ with $u_m(0,x)$ converging in $W^{1,2}(\Bbb R^2)$, then for any $\mu\in(0,T]$, there exists some constant $C(k,\mu,E_0,T)$ such that
$$\mathop {\sup }\limits_m \|D u_m\|_{W^{k,2}}\le C(k,\mu,T,E_0),$$
provided $\varepsilon(u_m,R_1)<\varepsilon_1$.
\end{corollary}
\begin{proof}
Following the proof of Lemma \ref{j.2} and iteration arguments, we can prove uniform bounds  $\|\nabla u\|_{\mathcal{W}^{k,2}}$ in  $m$. We omit the long but standard arguments. Then the desired result $\mathop {\sup }\limits_m \|D u_m\|_{W^{k,2}}\le C(k,\mu,T,E_0)$ is a consequence of the equivalence of $\|\nabla u\|_{\mathcal{W}^{k,2}}$ and  $\|D u\|_{W^{k,2}}$ when $k\ge2$.
\end{proof}

\begin{remark}\label{10.1}
It is easily seen from the above Lemmas that if $u$ is a regular solution to (\ref{1}) defined on $[0,T]\times \Bbb R^2$ with  $\varepsilon(R,T)<\varepsilon_1$, then $u$ can be extended to a regular solution on $[0,T_1]\times \Bbb R^2$, for some $T_1>T$.
\end{remark}

The following proposition is a corollary of the lemmas above whose proof is almost the same as heat flows of harmonic maps. Thus we will sketch the proof. The difference is that we need the outer ball energy estimate to ensure the compactness of approximate solutions because of the non-compactness of $\Bbb R^2$.
\begin{proposition}\label{a.1}
For any initial data $u_0\in W^{1,2}(\Bbb R^2;\mathcal{N})$, there exists a time $T(u_0)>0$ and a solution in $Y([0,T]\times\Bbb R^2)$ to (\ref{1}). Moreover, $T(u_0)$ is characterized by
$$
\mathop {\lim \sup }\limits_{T' \to T} \varepsilon (R,T') > {\varepsilon _1}, \mbox{  }\mbox{  }{\rm{for}}\mbox{  }{\rm{all}}\mbox{  }R\in(0,1].
$$
The solution is regular on $\Bbb R^2\times (0,\infty)$ with the exception of finitely many points $(x_l,T_l)$, $1\le l\le L$, characterized by
$$
\mathop {\lim \sup }\limits_{T' \to T_l} \int_{B_R(x_l)}|\nabla u(T',y)|^2 dy\le {\varepsilon _1}, \mbox{  }\mbox{  }{\rm{for}}\mbox{  }{\rm{all}}\mbox{  }R\in(0,1].
$$
\end{proposition}
\begin{proof}
Let $\{u_{m,0}\}$ be a sequence of regular initial data which approximate $u_0$ in $W^{1,2}(\Bbb R^2)$, this is possible by \cite{SU1}, \cite{SU2}. By the local theorem of \cite{KLPST}, (\ref{1}) admits a regular solution $u_m(t,x)$ with data $u_{m,0}$. Since $u_{m,0}$ converges to $u_0$, there exists $R>0$ sufficiently small such that
$$E(u_{m,0},B_{2R}(x))\le \varepsilon_1/2.$$
Lemma \ref{2.6} implies for $T_1$ of order $\varepsilon_1R^2$, we have for $t\in[0,T_1]$,
$$E(u_m(t),B_R(x))\le \varepsilon_1.$$
Applying Remark \ref{10.1}, Corollary \ref{j.4}, we have uniform bounds with respect to $m$ for
$\|u_m\|_{Y([0,T]\times\Bbb R^2)}$. We get from the compactness Lemma \ref{10.2} that there exists $u\in Y([0,T]\times\Bbb R^2)$ which is regular in $[\mu,T_1]$ for any $\mu>0$, satisfying (\ref{1}) in the weak sense.  The characterization of the singular time follows from corollary \ref{j.4}. The finiteness of singular points is due to the non-increasing of energy.
\end{proof}

The proof of Proposition \ref{a.1} given above yields more results than stated in Proposition \ref{a.1}. We summarize some useful results in the following proposition.
\begin{proposition}\label{energy}
Define the solution class $\mathcal{H}(I\times\Bbb R^2)$ as the set of all weak solutions to (\ref{1}) which satisfy for all $R>0$, $(s_1,s_2)\subset I$
\begin{align*}
&(i)\mbox{  }u\in Y(I\times\Bbb R^2) ;\\
&(ii)\mbox{ }\alpha \int_{{s_1}}^{{s_2}} {\left\| {{\partial _s}u} \right\|_{L_x^2}^2} + \alpha \int_{{s_1}}^{{s_2}} {\left\| \Sigma^2_{i=1}\nabla_i\partial_i u\right\|_{L_x^2}^2} ds \le \left(\left\| {\nabla u({s_1})} \right\|_{L_x^2}^2- \left\| {\nabla u({s_2})} \right\|_{L_x^2}^2 \right),\\
&\int^T_0\int_{\Bbb R^2}|\nabla u|^4dydt\lesssim \big(\mathop {\rm{esssup} }\limits_{0\le t\le T, x \in {\Bbb R^2}}E(u(t),B_R(x))\big)\big(\int^T_0\int_{\Bbb R^2}|\nabla^2u|^2dydt +\frac{T}{R^2}E(u_0)\big);\\
&(iii)\mbox{ }E(u({s_2});{B_R}(x)) \le E(u({s_1});{B_{2R}}(x)) + \frac{{{C_3}({s_2} - {s_1})}}{{{R^2}}}E({u_0}),\\
&\mbox{  }\mbox{  }\mbox{  }E(u({s_2});{B_{2R}}(x)) \ge E(u({s_1});{B_R}(x)) -(E(u(s_1))-E(u(s_2)))- \frac{{{C_3}({s_2} - {s_1})}}{{{R^2}}}E({u_0});\\
&(iv)\mbox{ }E(u(t)) {\rm{\mbox{ }is}} {\rm{\mbox{  }continuous \mbox{  }and \mbox{  }decreasing \mbox{  }with\mbox{  }respect\mbox{  } to}} \mbox{  }t\\
&(v)\mbox{  }\exists \mbox{  }{\rm{ classical}} \mbox{  }{\rm{soltuion }}\mbox{  }{u_n}\mbox{  }{\rm{ with }}\mbox{  } {\left\| {{u_n}(0,x) - {u_0}(x)} \right\|_{{W^{1,2}}}} \to 0,\mbox{  }{\partial _t}{u_n} \to {\partial _t}u\mbox{  }{\rm{ weakly}} \mbox{  }{\rm{in }}\mbox{  } L_{t,x}^2(I \times \Bbb R^2),\\
&\mbox{  }\mbox{  }{\rm{and }}\mbox{  }{\left\| {D{u_n} - Du} \right\|_{L_{t,x}^2(I \times \Bbb R^2)}} \to 0,\mbox{  }\left\| {{u_n} - u} \right\|_{C(I;L^2(\Bbb R^2))}\to 0.
\end{align*}
Then for any initial data $u_0\in W^{1,2}$, there exists a $T>0$ such that (\ref{1}) admits a weak solution $u(t,x)\in \mathcal{H}([0,T)\times\Bbb R^2)$. And the weak solution is unique as the limit of classical solutions in the following sense: If $u_1(t,x),u_2(t,x)\in Y([0,T]\times \Bbb R^2)$ are weak solutions to (\ref{1}) with initial data $u_0\in W^{1,2}$ and there exist classical solutions $\{u_n^1\}$, $\{u_n^2\}$ to (\ref{1}) which approximate $u_1$ and $u_2$ respectively in the sense of $(v)$. Particularly, for any initial data there exits a unique solution to (\ref{1}) in $\mathcal{H}([0,T]\times\Bbb R^2)$.
\end{proposition}
\begin{proof}
The existence of a solution which satisfies $(i),(v)$ is a direct corollary of the construction of the approximate solutions presented in Proposition \ref{a.1}.
From the proof of Proposition \ref{a.1}, we have $D u_n\to Du$ in $L^2([0,T];L^2(\Bbb R^2))$, thus we can assume $Du_n(t)\to D u(t)$ for almost all $t\in[0,T]$ in $L^2(\Bbb R^2)$. Furthermore, since $Y([0,T]\times\Bbb R^2)\subset C([0,T];{\dot{H}}^1)$, $u_n(0,x)\to u_0$ in $W^{1,2}$, we can prove $(ii),(iii)(iv)$ by first verifying them for a dense subset of $[0,T]$ then passing to all $t$ by the continuity of $u$ with respect to $t$ in $\dot{H}^1(\Bbb R^2)$.

It remains to prove the uniqueness. Suppose that $u_1(t,x),u_2(t,x)$ are two weak solutions to (\ref{1}) with initial data $u_0$ and there exit $\{u_n^1\}$, $\{u_n^2\}$ which are classical solutions to (\ref{1}) and approximate $u_1$, $u_2$ in the sense of $(v)$. By the extrinsic formulation (3.3), if we define $v^1_n=\iota\circ u_n^1,v^2_n=\iota\circ u_n^2$, then $w_n\triangleq v^1_n-v^2_n$ satisfies
\begin{align}\label{vb45}
\partial_t w_n=M(v^1_n)\big(d\Pi|_{\Pi v^1_n}(\Delta w_n)\big)+\big(M(v^1_n)(d\Pi|_{\Pi v^1_n})-M(v^2_n)(d\Pi|_{\Pi v^2_n})\big)(\Delta v^2_n).
\end{align}
Taking the inner product with $w_n$ on both sides of (\ref{vb45}), integration by parts, Lemma 3.1, Young's inequality and the compactness of $\mathcal{N}$ give
\begin{align}
\frac{d}{{dt}}\left\| {{w_n}} \right\|_{L_x^2}^2 &\le  - \alpha \left\| {\nabla {w_n}} \right\|_{L_x^2}^2 + \int_{\Bbb R^2} {\left| {{w_n}} \right|\left| {\nabla {w_n}} \right|} \left| {\nabla v_n^2} \right|dx + \int_{\Bbb R^2} {{{\left| {{w_n}} \right|}^2}\left| {\nabla {w_n}} \right|} \left| {\nabla v_n^2} \right|dx \\
&\le  - \frac{\alpha }{2}\left\| {\nabla {w_n}} \right\|_{L_x^2}^2 + C\int_{\Bbb R^2} {{{\left| {{w_n}} \right|}^2}\left( {{{\left| {\nabla v_n^2} \right|}^2} + {{\left| {\nabla v_n^1} \right|}^2}} \right)dx}.\label{z34sa}
\end{align}
By similar arguments as Lemma \ref{bb.1}, we have
\begin{align}\label{c56z}
{\left( {\int_0^T {\int_{\Bbb R^2} {{{\left| {{w_n}} \right|}^4}dxdt} } } \right)^{1/2}} \le \mathop {\sup }\limits_{t \in [0,T]} \int_{\Bbb R^2} {{{\left| {{w_n}} \right|}^2}dx}  + \int_0^T {\int_{\Bbb R^2} {{{\left| {\nabla {w_n}} \right|}^2}dx} dt}.
\end{align}
Then from Cauchy-Schwartz, (\ref{z34sa}), (\ref{c56z}), we obtain
\begin{align*}
\left\| {{w_n}(t)} \right\|_{L_x^2}^2 + \frac{\alpha }{2}\int_0^T {\left\| {\nabla {w_n}} \right\|_{L_x^2}^2ds \le } \left( {\mathop {\sup }\limits_{t \in [0,T]} \int_{\Bbb R^2} {{{\left| {{w_n}} \right|}^2}dx}  + \int_0^T {\int_{\Bbb R^2} {{{\left| {\nabla {w_n}} \right|}^2}dx} dt} } \right)\left\| {\nabla {W_n}} \right\|_{L_{x,t}^4}^2+\|w_n(0,x)\|_{L_x^2}^2
\end{align*}
where $|{\nabla {W_n}}|\triangleq |\nabla v^1_n|+|\nabla v^2_n|$.
Lemma \ref{6.1} implies for any $\eta>0$ there exists $\delta>0$ such that $\|\nabla v^2_n\|_{L_{x,t}^4(I'\times\Bbb R^2}^2+\|\nabla v^1_n\|_{L_{x,t}^4(I'\times\Bbb R^2}^2<\eta$ for $|I'|<\delta$. Let $\eta$ be sufficiently small, $T=\delta$, $t=t^*_n$ where $t^*_n$ achieves
$$\mathop {\sup }\limits_{t \in [0,\delta]} \int_{\Bbb R^2} {{{\left| {{w_n}} \right|}^2}dx},$$
then
$$
\left\| {{w_n}(t_n^*)} \right\|_{L_x^2}^2 + c\int_0^{\delta} \left\| {\nabla {w_n}} \right\|_{L_x^2}^2ds \le \|w_n(0,x)\|_{L_x^2}^2.
$$
Letting $n\to \infty$, we infer from $(v)$ that
$$\mathop {\sup }\limits_{t \in [0,\delta]} \int_{\Bbb R^2} {{{\left| {{w}} \right|}^2}dx}+\int^{\delta}_0\left\| \nabla u_1-\nabla u_2 \right\|_{L_x^2}^2dt=0.
$$
Hence we obtain $u_1(t)=u_2(t)$ in $L^2(\Bbb R^2)$ for all $t\in[0,\delta]$. Then the uniqueness in $[0,T]$ can be proved by the iteration due to $Y([0,T]\times \Bbb R^2)\subset C([0,T];L^2)$ and the decreasing of the energy.
\end{proof}

By an iteration argument and the non-increasing of energy, we have the global existence of weak solution.
\begin{proposition}\label{a.2}
For any initial data $u_0\in W^{1,2}(\Bbb R^2;\mathcal{N})$, there exists a global weak solution in $Y([0,\infty)\times\Bbb R^2)$ to (\ref{1}), which is regular on $\Bbb R^2\times (0,\infty)$ with the exception of finitely many points $(x_l,T_l)$, $1\le l\le L$, characterized by
$$
\mathop {\lim \sup }\limits_{t \to T_l} \int_{B_R(x_l)}|\nabla u(t,y)|^2dy \ge {\varepsilon _1}, \mbox{  }\mbox{  }{\rm{for}}\mbox{  }{\rm{all}}\mbox{  }R\in(0,1].
$$
\end{proposition}

The proof of the following bubbling theorem is standard, we omit the details.
\begin{proposition}\label{c.iio}
Let $u$ be the solution to (\ref{1}) constructed in Proposition \ref{a.2}, and suppose that $(x_0,T)$, $T\le \infty$, is a point where
$$
\mathop {\lim \sup }\limits_{t \uparrow T} {E(u(t);B_R(x_0))} > {\varepsilon _1}\mbox{ }\mbox{ }\mbox{ }\mbox{ }\forall R \in (0,1].
$$
Then there exist sequences $t_m\to T$, $x_m\to x_0$, $R_m\in (0,1]$, $R_m\to 0$ and a regular harmonic mapping $\omega:\Bbb R^2\to \mathcal {N}$ with $E(\omega)\ge \varepsilon_1$ such that as $m\to \infty$,
$$u(R_mx+x_m, t_m)\to \omega, \mbox{ }\mbox{ }\mbox{ }\mbox{ }{\rm{locally}}\mbox{  }{\rm{in}} \mbox{  }W^{2,2}(\Bbb R^2;\mathcal{N}).
$$
\end{proposition}

\section{Acknowledgments}
The authors thank Professor Youde Wang and Hao Yin for helpful discussions and encouragements.

\end{document}